\documentclass[12pt]{amsart}

\usepackage{hyperref}
\hypersetup{
    colorlinks=true,
    linkcolor=red,
    citecolor=red,
    urlcolor=blue,
}

\makeatletter
\newcommand\org@maketitle{}
\newcommand\@authors{}
\let\org@maketitle\maketitle
\def\maketitle{%
	\let\@authors\authors
	\nxandlist{; }{ and }{; }\@authors
	\hypersetup{
		linktocpage=true,
		pdftitle={\@title},
                pdfauthor={\@authors},
                pdfsubject={\subjclassname. \@subjclass},
		pdfkeywords={\@keywords},
	}%
	\org@maketitle
}
\makeatother
\usepackage{comment}
\usepackage{amssymb}
\usepackage{rsfso}
\usepackage{mathtools}
\usepackage{microtype}
\usepackage[alphabetic,msc-links]{amsrefs}
\usepackage{enumitem}
\usepackage{amsfonts}
\usepackage{color}

\usepackage{doi}
\renewcommand{\PrintDOI}[1]{\doi{#1}}

\let\arXiv\arxiv

\usepackage[
  hmarginratio={1:1},     % equal left and right margins
  vmarginratio={1:1},     % equal top and bottom margins
  textwidth=17cm,        % new text width
  textheight=21cm,
  heightrounded,          % always useful
]{geometry}

\numberwithin{equation}{section}

\newtheorem{theorem}{Theorem}[section]
\newtheorem{teo}[theorem]{Theorem}
\newtheorem{lem}[theorem]{Lemma}
\newtheorem{cor}[theorem]{Corollary}
\newtheorem{pro}[theorem]{Proposition}
\theoremstyle{definition}
\newtheorem{defi}[theorem]{Definition}

\theoremstyle{remark}
\newtheorem{oss}[theorem]{Remark}

\newcommand{\e}{\varepsilon}

\newcommand{\R}{\mathbb{R}}
\newcommand{\N}{\mathbb{N}}

\newcommand{\D}{\nabla}

\newcommand{\supp}{\operatorname{spt}}
\renewcommand{\div}{\operatorname{div}}
\newcommand{\loc}{\mathrm{loc}}
\newcommand{\dist}{\operatorname{dist}}

\def\XXint#1#2#3{{\setbox0=\hbox{$#1{#2#3}{\int}$}
    \vcenter{\hbox{$#2#3$}}\kern-.5\wd0}}

\DeclareRobustCommand{\rchi}{{\mathpalette\irchi\relax}}
\newcommand{\irchi}[2]{\raisebox{\depth}{$#1\chi$}}

\mathchardef\ordinarycolon\mathcode`\:
\mathcode`\:=\string"8000
\begingroup \catcode`\:=\active
  \gdef:{\mathrel{\mathop\ordinarycolon}}
\endgroup

\author{Gabriele Fioravanti}
\address{Gabriele Fioravanti\newline\indent
Dipartimento di Matematica "Giuseppe Peano"
\newline\indent
Universit\`a degli Studi di Torino
\newline\indent
Via Carlo Alberto 10, 10124, Torino, Italy}
\email{gabriele.fioravanti@unito.it}

\title[The Dirichlet problem on lower dimensional boundaries: Schauder estimates]{The Dirichlet problem on lower dimensional boundaries: Schauder estimates via perforated domains}

\thanks{\emph{Acknowledgement}. The author is member of the INDAM (“Istituto Nazionale di Alta Matematica”) research group GNAMPA} %G.F. is supported by the GNAMPA-INDAM project \emph{Teoria della regolarit\`a per problemi ellittici e parabolici con diffusione anisotropa e pesata}, CUP\_E53C22001930001.}

\allowdisplaybreaks

\begin{document}
\begin{abstract} 
In this paper, we investigate the Dirichlet problem on lower dimensional manifolds for a class of weighted elliptic equations with coefficients that are singular on such sets. Specifically, we study the problem  
\[\begin{cases}
-\div(|y|^a A(x,y) \nabla u) = |y|^a f + \div(|y|^a F), \\
u = \psi, \quad \text{ on } \Sigma_0,
\end{cases}
\]  
where \((x,y) \in \mathbb{R}^{d-n} \times \mathbb{R}^n\), \(2 \leq n \leq d\), \(a + n \in (0,2)\), and \(\Sigma_0 = \{|y| = 0\}\) is the lower dimensional manifold where the equation loses uniform ellipticity.  

Our primary objective is to establish \(C^{0,\alpha}\) and \(C^{1,\alpha}\) regularity estimates up to \(\Sigma_0\), under suitable assumptions on the coefficients and the data. Our approach combines perforated domain approximations, Liouville-type theorems and a fine blow-up argument. 
\end{abstract}

\maketitle

\section{Introduction}

Let $2 \le n \le d$ be two integers and $z = (x,y) \in \R^{d-n} \times \R^n$. Let us define the lower dimensional manifold $\Sigma_0 := \{(x,y) \in \R^d : |y| = 0\}$, which has dimension $d-n$, and the weight $|y|^a = \dist_{\Sigma_0}^a(z)$, where the real parameter $a$ satisfies $a + n \in (0,2)$. We study the following equation
\begin{equation}\label{eq:1}
\begin{cases}
-\div(|y|^a A \nabla u)= |y|^a f+\div(|y|^aF), & \text{in }B_1\setminus \Sigma_{0},\\
u=\psi, & \text{on } \Sigma_0\cap B_1,
\end{cases}
\end{equation}
where $B_1\subset\R^d$ denotes the unit ball with center at $0$, $A:B_1\to\R^{d,d}$ is a symmetric $d$-dimensional matrix satisfying the following ellipticity condition
\begin{equation}\label{eq:unif:ell}
    \lambda|\xi|^2\le A(z)\xi\cdot\xi\le\Lambda|\xi|^2,
\end{equation}
for all $\xi\in\mathbb{R}^{d}$ and a.e. $z\in B_1$, where $0 < \lambda \le \Lambda < \infty$ are fixed constants. The terms $f:B_1\to\R$, $F:B_1\to\R^d$ and $\psi:\Sigma_0\cap B_1\to\R$ belong to suitable spaces, which will be introduced later. 
We notice that, when \(n = d\), the lower-dimensional boundary \(\Sigma_0 = \{0\}\) reduces to a single point. In this case, we assume the boundary condition is simply \(u(0) = 0\).
The operators $\D$ and $\div$ denote the gradient and the
divergence with respect to the variable $z$, respectively.
Weak solutions to this equation are naturally defined within the framework of weighted Sobolev spaces, which will be introduced in Section \ref{section:sobolev}. Thus, we say that $u$ is a weak solution to \eqref{eq:1} if $u\in H^{1,a}(B_1):=H^{1}(B_1,|y|^a dz)$, satisfies
\[
\int_{B_1}|y|^a A\D u\cdot \D\phi dz= \int_{B_1}|y|^a (f\phi-F\cdot\D\phi)dz,
\]
for every $\phi\in C_c^\infty(B_1\setminus\Sigma_0)$ and $u=\psi$ in the sense of the trace (see Definition \ref{def:weak:sol:trace:0}).

Our primary goal is to establish local regularity estimates up to $\Sigma_0$ for weak solutions to \eqref{eq:1}. Specifically, we show that,  under suitable assumptions on the data, solutions are $C^{0,\alpha}(B_{1/2})$ and, in some cases, may be $C^{1,\alpha}(B_{1/2})$.
As we will see later, our results are \emph{sharp} with respect to the assumptions on the data. Specifically, the function $|y|^{2-a-n}$ is a solution to \eqref{eq:1}, when $A=I$, $f=0$, $F=0$ and $\psi=0$ and it is $C^{0,2-a-n}$ if $a+n\in[1,2)$ and $C^{1,1-a-n}$ if $a+n\in(0,1)$. The idea behind our theorems is that this solution is the \emph{worst regular} solution to \eqref{eq:1} when $A=I$, $f=0$, $F=0$ and $\psi=0$.

The Dirichlet boundary condition $u=\psi$ on $\Sigma_0$ requires some clarifications. For general values of $a\in\R$, the trace of functions which belongs to $H^{1,a}(B_1)$ on $\Sigma_0$ might not be well defined. In the paper \cite{Nek93}, the author shows the existence of a trace operator for a large class of weighted Sobolev spaces on lower dimensional boundaries. Specifically, the following result holds.
\begin{teo}\cite{Nek93}*{Theorem 2.3}\label{teo:NEK}
Let $\Gamma$ be a $(d-n)$-dimensional $C^1$-manifold, $\dist_\Gamma$ be the distance from $\Gamma$, $a+n\in(0,2)$. Then, there exits a unique bounded linear operator
\[
T: H^{1}(B_1,\dist_\Gamma^a) \to L^2({\Gamma\cap B_1})
\]
such that
\[
Tu=u_{|_\Gamma},
\]
for every $u \in C^\infty(\overline{B_1})$.
\end{teo}
Hence, the restriction we impose on the parameter $a+n\in(0,2)$ ensures that the boundary condition $u=\psi$ on $\Sigma_0$ makes sense. In other words, introducing the weight as a power of the distance to the boundary provides a natural framework for studying the Dirichlet problem for linear elliptic operators on lower-dimensional boundaries. Without such a weight, the solutions do not "see" the lower-dimensional sets due to capacitary reasons: for instance, a harmonic function in \( B_1 \setminus \Sigma_0 \) is the same as a harmonic function in the whole \( B_1 \).

We also emphasize that in the forthcoming work \cite{CFV24}, Cora, Vita and the author will systematically study the same equation in the broader case \(a+n > 0\), with a focus on solutions which satisfies an homogeneous \emph{conormal boundary condition} on \(\Sigma_0\) when \(a+n \in (0,2)\).

\medskip

The study of such equations falls within the theory of non uniformly elliptic operators, as the presence of the singular weight causes the operator's coefficients to blow up on the manifold $\Sigma_0$. In the seminal paper \cite{FKS82}, the authors extended the De Giorgi-Nash-Moser theory to weighted elliptic equations, where the weight arises from quasi-conformal mappings or belongs to the Muckenhoupt $A_2$ class. In particular, under suitable assumptions, they proved the validity of the Harnack inequality, ensuring H\"older regularity of solutions with a non explicit H\"older exponent. Along this line, we also refer to the work \cite{HeiKilMar06}. Our weight $|y|^a$ belongs to the Muckenhoupt $A_2$ class when $a + n \in (0, 2n)$, so our assumption $a + n \in (0, 2)$ ensures that the known results for this class of weights apply, guaranteeing that the solutions to our problem satisfy some H\"older estimates. However, the peculiar geometry of the singular set of our weight $|y|^a$, combined with its homogeneity property, allows us to obtain more refined results compared to the general theory mentioned above. 

In recent years, there have been significant contributions to the study weighted elliptic equations, where the weight behaves like the power of the distance from the boundary of a set. The most notable case is $n=1$, which is closely related to the extension theory for fractional operators developed by Caffarelli and Silvestre in their seminal paper \cite{CafSil07} (we refer also to \cite{mar}, where the authors study fractional operators in conformal geometry). This setting is now well understood, and there is a rich literature on the regularity properties of such equations.
In particular, we highlight the works \cites{SirTerVit21a, SirTerVit21b, TerTorVit22, DongVita}, where the authors establish a complete Schauder regularity theory for degenerate/singular elliptic equations, and its parabolic counterpart \cites{AFV24, AFV24b}. Additionally, we mention \cites{Dong1, Dong2}, where elliptic and parabolic weighted equations are studied, yielding alternative regularity results.

In the paper \cite{DavFenMay21}, David, Feneuil and Mayboroda developed an elliptic theory for equations which are degenerate/singular on lower dimensional boundaries. In particular, in \cite{May}, the authors extensively studied our operator in the case \( a+n=1 \), under weaker assumptions on the coefficients \( A \). They proved the solvability of the Dirichlet problem in this setting. We also refer to \cites{DavMay22} and the references therein for a broader overview of this topic.
Moreover, our class of operators also arises in the context of singular harmonic maps to study equations related to black holes (see \cites{Wei90,Wei92,LiTia91,LiTia92,LiTia93}). In particular, in \cite{Ngu11}, the author highlights the connection between these singular harmonic maps and differential operators like ours, in the case where $a < 0$. Additionally, our work is related to Mazzeo's theory of edge operators \cites{Maz1, Maz2}, which emerges in boundary problems with higher codimensional boundaries and provides essential insights into solution regularity by establishing Fredholmness in degenerate H\"older or Sobolev spaces. 

Finally, it seems that our operator could be a good model for free boundary problems of the obstacle type (see, for example, the classical papers \cites{Caf77, Caf79}), where the obstacle is very thin (with dimension less than \( d-2 \)), see also \cite{fern}. As already noted above, classical elliptic operators do not allow for such problems, as they cannot "see" small sets like these due to capacitary reasons.

\subsection*{Main results} 

The main goal of this paper is to establish local $C^{0,\alpha}$ and $C^{1,\alpha}$ regularity estimates up to the singular set $\Sigma_0$ for weak solutions to \eqref{eq:1}. These results are presented in two main theorems: the first provides H\"older estimates for the solutions, while the second establishes H\"older estimates for the gradient.

\begin{teo}\label{teo:0}
Let $2\le n\le d$, $a+n\in(0,2)$, $p>{d}/{2}$, $q>d$ and
\begin{equation}\label{eq:alpha:0}
\alpha\in (0,2-a-n)\cap(0,2-{d}/{p}]\cap(0,1-{d}/{q}]\cap(0,1).
\end{equation}
Let $A$ be a continuous matrix satisfying \eqref{eq:unif:ell} and $\|A\|_{C^0({B_1})}\le L$, $f\in L^{p,a}(B_{1})$, $F\in L^{q,a}(B_{1})^d$ and $\psi\in C^{0,1}(\Sigma_0\cap B_1)$.
Let $u$ be a weak solution to \eqref{eq:1}.

Then, $u\in C^{0,\alpha}(B_{1/2})$ and there exists a constant $c>0$, depending only on $d$, $n$, $a$, $\lambda$, $\Lambda$, $p$, $q$, $\alpha$ and $L$ such that
\begin{equation}\label{eq:c0}
\|u\|_{C^{0,\alpha}(B_{1/2})}  \le c\big(
\|u\|_{L^{2,a}(B_{1})}
+\|f\|_{L^{p,a}(B_{1})}
+\|F\|_{L^{q,a}(B_{1})^d}
+\|\psi\|_{C^{0,1}(\Sigma_0\cap B_1)}
\big).
\end{equation}
\end{teo}

\begin{teo}\label{teo:1}
Let $2\le n\le d$, $a+n\in(0,1)$, $p>{d}$ and
\begin{equation}\label{eq:alpha:1}
\alpha\in (0,1-a-n)\cap (0,1-{d}/{p}].
\end{equation}
Let $A$ be a $\alpha$-H\"older continuous matrix satisfying \eqref{eq:unif:ell} and $\|A\|_{C^{0,\alpha}{(B_1)}}\le L$, $f\in L^{p,a}(B_1)$, $F\in C^{0,\alpha}(B_1)$ and $\psi\in C^{1,\alpha}(\Sigma_0\cap B_1)$. Let $u$ be a weak solution to \eqref{eq:1}.

Then, $u\in C^{1,\alpha}(B_{1/2})$ and there exists a constant $c>0$, depending only on $d$, $n$, $a$, $\lambda$, $\Lambda$, $p$, $\alpha$ and $L$ such that
\begin{equation}\label{eq:c1}
\|u\|_{C^{1,\alpha}(B_{1/2})}  \le c\big(
\|u\|_{L^{2,a}(B_{1})}
+\|f\|_{L^{p,a}(B_{1})}
+\|F\|_{C^{0,\alpha}(B_1)}
+\|\psi\|_{C^{1,\alpha}(\Sigma_0\cap B_1)}
\big).
\end{equation}
In addition, $u$ satisfies the following boundary condition
\begin{equation}\label{eq:BC:conormal}
\begin{cases}
\D_x u(x,0)=\D_x\psi(x,0), \\
(A\D u +F) (x,0) \cdot e_{y_i}= 0,
\end{cases} \text{ for every } (x,0)\in \Sigma_0\cap B_{1/2}, \text{ and } i=1,\dots,n.
\end{equation}
\end{teo}

Before presenting the proof idea, we recall a method developed in \cites{Soave,SirTerVit21a} to establish regularity estimates for weighted elliptic equations that are degenerate or singular on a set of codimension one, that is, when \(n=1\). This case, where the singular set is an hyperplane of dimension \(d-1\), does not fall under the scope of our study, which focuses on \(n \geq 2\). 
The authors develop a regularity theory for such equations by regularizing the weight and using an approximation argument: they first establish $\varepsilon$-stable regularity results for solutions with the regularized weight $(\varepsilon^2 + y^2)^{a/2}$, and then take the limit as $\varepsilon \to 0$.
However, this approach seems not to work in our case, where $n\ge2$. Indeed, by regularizing the weight, that is, considering $\rho_\e^a(y):=(\varepsilon^2 + |y|^2)^{a/2}$ and solutions to uniformly elliptic problems with this type of weight, the $H^1(\rho_\e^a)$-capacity of $\Sigma_0$ will be zero. As we discussed earlier, in this situation, the \emph{classical} Sobolev spaces will fail to capture information about the boundary condition on the lower dimensional set $\Sigma_0$. Consequently, the approximation problem with this type of weight loses critical information about the boundary condition, making it impossible to recover uniform regularity estimates for the weighted problem.

We adopt a different approach based on perforated domains, aiming to establish uniform estimates in this framework. For small \( 0 < \varepsilon \ll 1 \), we define \( \Sigma_\varepsilon := \{|y| \le \varepsilon\} \) as the \( \varepsilon \)-neighbourhood of \( \Sigma_0 \), noting that its boundary $\partial \Sigma_\e =\{|y|=\e\}$ has dimension $d-1$ and $\partial \Sigma_\e \to \Sigma_ 0$ as $\e\to 0$. The central idea is to approximate the lower dimensional boundary, which has dimension $d-n$, with a \emph{classical} boundary of dimension $d-1$ and impose the Dirichlet condition on this set. Hence, we consider solutions to the problem
\[
\begin{cases}
-\div(|y|^a A \nabla u) = |y|^a f + \div(|y|^a F), & \text{in } B_1 \setminus \Sigma_\varepsilon, \\
u = 0, & \text{on } \partial \Sigma_\varepsilon \cap B_1.
\end{cases}
\]
In this context, since the singular set $\Sigma_0$ is sufficiently far from \( B_1 \setminus \Sigma_\varepsilon \), we can apply classical regularity theory to obtain \( C^{0,\alpha}(B_{1/2} \setminus \Sigma_\varepsilon) \) and \( C^{1,\alpha}(B_{1/2} \setminus \Sigma_\varepsilon) \) regularity estimates, with constants that may depend on \( \varepsilon \). The crucial step is proving that these estimates are uniform as \( \varepsilon \to 0 \), as shown in Theorems \ref{teo:0alpha:e} and \ref{teo:1alpha:e}. Once we have these uniform estimates, we employ an approximation argument (see Section \ref{s:approximation}) to pass to the limit as \( \varepsilon \to 0 \), thereby recovering the desired regularity results for the original problem.

We emphasize that proving \( C^{1,\alpha} \) regularity for solutions to \eqref{eq:1} requires a refined approach. Specifically, the uniform estimates in perforated domains require an additional assumption on the field \( F \), as highlighted in Remark \ref{R:F:radial}. Consequently, proving the main theorem necessitates a double approximation strategy: the first via perforated domains and the second through a standard mollification argument. This is combined with some \emph{a priori} estimates for solutions to \eqref{eq:1} under an additional boundary condition on \( \Sigma_0 \), as detailed in Proposition \ref{P:a:priori}.

The strategy for establishing the $\e$-uniform estimates is based on a contradiction argument combined with a blow-up procedure, drawing inspiration from Simon's work \cite{Sim95}. A key component of this approach is the following Liouville-type theorem, which applies to entire solutions satisfying a specific growth-control condition at infinity.

\begin{teo}\label{teo:liouville}
Let $2\le n\le d$, $a+n\in(0,2)$, $\e\ge0$. Let $A$ be a constant symmetric matrix satisfying \eqref{eq:unif:ell} and $u$ be an entire solution to 
\begin{equation}\label{eq:entire:solution}
\begin{cases}
-\div(|y|^a  A\nabla {u})= 0, & \text{in }\R^d\setminus \Sigma_{\e},\\
u=0, & \text{on } \partial\Sigma_\e,
\end{cases}
\end{equation}
(see Definition \ref{def:weak:sol}).
Assume that there exist constants $c>0$, $\gamma\in (0,2-a-n)$ such that
\begin{equation}\label{eq:growth}
|u(z)|\le c(1+|z|^\gamma), \quad \text{for a.e. }z\in\R^d\setminus \Sigma_{\e},
\end{equation}
Then, $u$ is identically zero.

\end{teo}

Finally, as a consequence of our main theorems, we extend our results to more general weighted equations, where the weight $\delta$ behaves like a distance function from a regular manifold $\Gamma$ of dimension $d-n$ (see Definition \ref{def:rho}). Specifically, in Corollaries \ref{cor:0} and \ref{cor:1}, we are able to prove local \( C^{0,\alpha}(B_{1/2} ) \) and \( C^{1,\alpha}(B_{1/2}) \) regularity estimates for weak solutions to the following equation 
\begin{equation}\label{eq:1:cor}
\begin{cases}
-\div(\delta^a A \nabla u)= \delta^a f+\div(\delta^aF), & \text{in }B_1\setminus \Gamma,\\
u=\psi, & \text{on } \Gamma\cap B_1.
\end{cases}
\end{equation}
The precise definition of solutions to \eqref{eq:1:cor} will be given later in Section \ref{section:curved}.

\subsection*{Structure of the paper}

The paper is organized as follows. In Section \ref{section:setting}, we set up the problem by introducing weighted Sobolev spaces, discussing their basic properties, and providing the definition of weak solutions to our equation along with some preliminary results, including the approximation Lemma \ref{lem:approximation}. Section \ref{section:liouville} is devoted to the proof of the Liouville-type Theorem \ref{teo:liouville}. In Sections \ref{section:holder} and \ref{section:schauder}, we establish the main results, Theorems \ref{teo:0} and \ref{teo:1}, which concern $C^{0,\alpha}$ and $C^{1,\alpha}$ regularity, respectively. Lastly, in Section \ref{section:curved}, we extend these results to solutions of the more general equation \eqref{eq:1:cor}.

\section{Functional setting and preliminary results}\label{section:setting}

\subsection{Weighted Sobolev Spaces}\label{section:sobolev}

Let $a+n\in(0,2)$, $R>0$ and $B_R:=\{z\in\R^d: |z|<R\}$ be the ball centered in $0$ and radius $R$.
We define the weighted Lebesgue spaces
\[L^{p,a}(B_R):=L^p(B_R,|y|^a dz),\]
and for vector field
\[L^{p,a}(B_R)^d:=L^p(B_R,|y|^a dz)^d.\]
The Sobolev space $H^{1,a}(B_R)$ is defined as the completion of $C^\infty(\overline{B_R})$ with respect to the norm
\begin{equation}\label{eq:Norm:omega}
\|u\|_{H^{1,a}(B_R)}=\Big(\int_{\Omega} |y|^a u^2 dz +\int_{B_R} |y|^a |\D u|^2 dz\Big)^{1/2},
\end{equation}
where $C^\infty(\overline{B_R})=\{u_{|_{B_R}}: u \in C_c^\infty(\R^d)\}$. 

Since we are interested in functions which vanish on $\Sigma_0 = \{|y|=0\}$, we define the Sobolev space $\tilde H^{1,a}(B_R)$ as the completion of $C_c^\infty(\overline{B_R}\setminus\Sigma_0)$ with respect to the norm 
$\|\cdot\|_{H^{1,a}(B_R)}$.

Additionally, we define the Sobolev space $ H^{1,a}_0(B_R)$ as as the completion of $C_c^\infty({B_R}\setminus\Sigma_0)$ with respect to the norm $\|\cdot\|_{H^{1,a}(B_R)}$, which contains functions having zero trace on $\partial (B_R\setminus\Sigma_0)$.
\subsection{Weighted Sobolev Spaces in perforated domains}\label{section:sobolev:buco}
In the rest of the paper, we use the notation $0<\e\ll 1$ to indicate that $\e$ is a small positive number. Let us define
\[\Sigma_\e:=\{ (x,y):|y|\le \e  \},\quad\partial\Sigma_\e:=\{(x,y):|y|= \e  \}.\]
We set
\[L^{p,a}(B_R\setminus\Sigma_\e):=L^p(B_R\setminus\Sigma_\e,|y|^a dz),\]
and
\[L^{p,a}(B_R\setminus\Sigma_\e)^d:=L^p(B_R\setminus\Sigma_\e,|y|^a dz)^d.\]
Let us define the norm
\[\|u\|_{H^{1,a}(B_R\setminus\Sigma_\e)}=\Big(\int_{B_R\setminus \Sigma_\e} |y|^a u^2 dz +\int_{B_R\setminus \Sigma_\e} |y|^a |\D u|^2 dz\Big)^{1/2}.\]
We define Sobolev spaces in perforated domains as follows.
\begin{enumerate}
\item [$\bullet$] $ {H}^{1,a}(B_R\setminus\Sigma_\e)$ as the completion of $C^\infty(\overline{B_R\setminus\Sigma_\e})$ w.r.t. the norm $\|\cdot\|_{H^{1,a}(B_R\setminus\Sigma_\e)}$,
\item [$\bullet$] $\tilde {H}^{1,a}(B_R\setminus\Sigma_\e)$ as the completion of $C_c^\infty(\overline{B_R}\setminus\Sigma_\e)$ w.r.t. the norm $\|\cdot\|_{H^{1,a}(B_R\setminus\Sigma_\e)}$,
\item [$\bullet$] ${H}^{1,a}_0(B_R\setminus\Sigma_\e)$ as the completion of $C_c^\infty({B_R\setminus\Sigma_\e})$ w.r.t. the norm $\|\cdot\|_{H^{1,a}(B_R\setminus\Sigma_\e)}$.
\end{enumerate}
When $\e=0$, we identify the spaces 
\[
L^{p,a}(B_R\setminus\Sigma_0)=L^{p,a}(B_R),\quad \tilde{H}^{1,a}(B_R\setminus\Sigma_0)= \tilde{H}^{1,a}(B_R), \quad {H}^{1,a}_0(B_R\setminus\Sigma_0)= {H}_0^{1,a}(B_R).
\]
Moreover, since the Poincaré inequality holds true (see Proposition \ref{P:basic}) in $\tilde{H}^{1,a}(B_R\setminus\Sigma_\e)$ for every $0\le \e\ll 1$, we have that 
\[\|u\|_{H^{1,a}_0(B_R\setminus\Sigma_\e)}=\Big(\int_{B_R\setminus \Sigma_\e} |y|^a |\D u|^2 dz \Big)^{1/2},\]
defines an equivalent norm to $\|\cdot\|_{H^{1,a}(B_R\setminus\Sigma_\e)}$ in $\tilde H^{1,a}(B_R\setminus\Sigma_\e)$.

\begin{oss}\label{rem:inclusion}
For \(\varepsilon > 0\), functions in \(\tilde{H}^{1,a}(B_R \setminus \Sigma_{\varepsilon})\) can be identified with their trivial extensions in the whole \(B_R\). Consequently, we have the following inclusion of spaces:
\[
\tilde{H}^{1,a}(B_R \setminus \Sigma_{\varepsilon}) \subset \tilde{H}^{1,a}(B_R).
\]
We note that this result is frequently used throughout the paper, particularly in the analysis of approximations on perforated domains.
\end{oss}

The following Proposition establishes several fundamental inequalities in the space $\tilde H^{1,a}(B_R\setminus\Sigma_\e)$, namely the Hardy inequality, the Poincaré inequality, the Poincaré trace inequality, and a Sobolev-type inequality, which are uniform with respect to the parameter $\e$.

\begin{pro}\label{P:basic}
Let $2\le n \le d$, $a+n\in(0,2)$, $R>0$ and $0\le\e\ll 1$. Then, there exist a constant $c>0$, depending only on $d$, $n$, $a$ and $R$ such that 
\begin{align}
&\int_{B_R}|y|^a \frac{u^2}{|y|^2}dz\le c\int_{B_R}|y|^a|\D u|^2dz,\label{eq:hardy}\\
&\int_{B_R}|y|^a u^2 dz \le c \int_{B_R}|y|^a |\nabla u|^2 dz,\label{eq:poincaré}\\
&\int_{ \partial B_R}|y|^a u^2 d\sigma\le c \int_{B_R}|y|^a |\nabla u|^2 dz,\label{eq:poincaré:trace}\\
&\Big(\int_{B_R}|y|^{a} |u|^{2^*} dz\Big)^{2/2^*} dz \le c \int_{B_R}|y|^a |\nabla u|^2 dz,\label{eq:embedding}
\end{align}
for every $u\in C_c^\infty(\overline{B_R}\setminus\Sigma_\e)$. In the last inequality $2^*:=2d/(d-2)$ if $d>2$ and $2^*$ can be replaced by any $p \in [1,\infty)$ if $d=2$, and in this case, the constant $c>0$ also depends on $p$.
\end{pro}

\begin{proof}

The Hardy inequality \eqref{eq:hardy} is well known nowadays, see, for example \cite{Maz11}.

The Poincaré inequality \eqref{eq:poincaré} immediately follows by the validity of the Hardy inequality \eqref{eq:hardy}, in fact
\[
\int_{B_R}|y|^a u^2 dz\le c \int_{B_R}|y|^a \frac{u^2}{|y|^2}dz,
\]
for some $c>0$ depending only on $R$.

By using the classical embedding $H^1(B_R) \hookrightarrow L^2(\partial B_R)$ to the function $|y|^{a/2} u \in C_c^\infty(\overline{B_R}\setminus\Sigma_\e)$, the Hardy inequality \eqref{eq:hardy}, combined with H\"older and Young inequalities, we have that
\begin{align*}
\int_{\partial B_R} |y|^{a}u^2 d\sigma &\le c \int_{ B_R}|\D(|y|^{a/2}u)|^2 dz\le c \int_{ B_R} \Big(|y|^{a}|\D u|^2+ \frac{a^2}{4}|y|^a\frac{u^2}{|y|^2}+a|y|^a\frac{|u|}{|y|}|\D u|  \Big) dz \\
&\le  c \int_{ B_R} |y|^{a}|\D u|^2 dz,
\end{align*}
and \eqref{eq:poincaré:trace} holds.

Finally, let's prove the Sobolev embedding \eqref{eq:embedding}. By using $a<0$, the classical Sobolev embedding $H^1(B_R) \hookrightarrow L^{2^*}( B_R)$ to the function $|y|^{a/2} u \in C_c^\infty(\overline{B_R}\setminus\Sigma_\e)$, the Hardy inequality \eqref{eq:hardy}, the Poincaré inequality \eqref{eq:poincaré}, combined with H\"older and Young inequalities, we obtain
\begin{align*}
&\Big(\int_{B_R}|y|^a|u|^{2^*}dz\Big)^{2/2^*}\le c \Big(\int_{B_R}\big(|y|^{a/2}|u|\big)^{2^*}dz\Big)^{2/2^*} \le c \int_{B_R} \Big( |y|^a u^2+|\D(|y|^{a/2}u)|^2\Big)dz\\
&\le   c \int_{B_R} \Big( |y|^a u^2+  |y|^{a}|\D u|^2+ \frac{a^2}{4}|y|^a\frac{u^2}{|y|^2}+a|y|^a\frac{|u|}{|y|}|\D u|  \Big) dz
\le c\int_{B_R}|y|^a|\D u|^2 dz.
\end{align*}
Hence, the proof is complete.
\end{proof}

\subsection{Weak solutions}
In this section we give the definition of weak solutions.

\begin{defi}\label{def:weak:sol}
Let $2\le n\le d$, $a+n \in (0, 2)$, $R>0$ and $0\le \e\ll 1$. Let $A$ be matrix satisfying \eqref{eq:unif:ell}, $f\in L^{2,a}(B_R\setminus\Sigma_{\e})$ and $F\in L^{2,a}(B_R\setminus\Sigma_{\e})^d$. We say that $u$ is a weak solution to
\begin{equation}\label{eq:weak:sol}
\begin{cases}
-\div(|y|^a A \nabla u)= |y|^a f+\div(|y|^aF), & \text{in }B_R\setminus \Sigma_{\e},\\
u=0, & \text{on } \partial\Sigma_\e\cap B_R,
\end{cases}
\end{equation}
if $ u\in \tilde H^{1,a}(B_R\setminus\Sigma_\e)$ and satisfies
\begin{equation}\label{eq:weak:sol:integral}
\int_{B_R}|y|^a A\D u\cdot \D\phi dz = \int_{B_R}|y|^a (f\phi-F\cdot\D\phi) dz,
\end{equation}
for every $\phi\in C_c^\infty(B_R\setminus\Sigma_\e).$

\smallskip

We say that $u$ is an entire solution to
\begin{equation*}
\begin{cases}
-\div(|y|^a A \nabla u)= |y|^a f+\div(|y|^aF), & \text{in }\R^d\setminus \Sigma_{\e},\\
u=0, & \text{on } \partial\Sigma_\e,
\end{cases}
\end{equation*}
if $u$ is a weak solution to \eqref{eq:weak:sol} for every $R>0$.
\end{defi}

\begin{oss}\label{rem:unicity}
Using the validity of the Poincaré inequality \eqref{eq:poincaré}, we have existence and uniqueness for solutions to \eqref{eq:weak:sol} which also satisfy a boundary condition on $\partial B_R\setminus\Sigma_{\e}$. In fact, if $u$ is a weak solution to \eqref{eq:weak:sol} satisfying $u-\bar u \in H^{1,a}_0(B_R\setminus\Sigma_\e) $, for some $\bar u\in \tilde H^{1,a}(B_R\setminus\Sigma_\e) $, then $u$ is a minimizer to the functional
\[
J(v):=\int_{B_R\setminus\Sigma_\e} |y|^a \Big(\frac{A\D v\cdot\D v}{2}-fv+F\cdot\D v\Big)dz,
\]
over  
\[X:=\{v\in H^{1,a}(B_R\setminus\Sigma_\e):v-\bar u\in H^{1,a}_0(B_R\setminus\Sigma_\e) \},\]
and $J$ is coercive. By a standard application of the Weierstrass Theorem, we have existence and uniqueness of solutions to \eqref{eq:weak:sol} with prescribed trace on $\partial B_R \setminus\Sigma_{\e}$.
\end{oss}

When $\e=0$, recalling the trace Theorem \ref{teo:NEK}, we also give a definition of weak solutions with prescribed trace on the lower dimensional boundary $\Sigma_0$.

\begin{defi}\label{def:weak:sol:trace:0}
Let $2\le n\le d$, $a+n\in(0,2)$, $R>0$ and $A$ satisfies \eqref{eq:unif:ell}. Let $f\in L^{2,a}(B_R)$, $F\in L^{2,a}(B_R)^d$ and $\psi\in  L^{2}( \Sigma_0 \cap B_R)$. We say that $u$ is a weak solution to
\begin{equation}\label{eq:weak:sol:trace:0}
\begin{cases}
-\div(|y|^a A \nabla u)= |y|^a f+\div(|y|^aF), & \text{in }B_R\setminus \Sigma_{0},\\
u=\psi, & \text{on } \Sigma_0\cap B_R,
\end{cases}
\end{equation}
if $u\in {H}^{1,a}(B_R)$, satisfies \eqref{eq:weak:sol:integral} for every $\phi\in C_c^\infty(B_R\setminus\Sigma_0)$ and $u=\psi$ on $\Sigma_0\cap B_R$, in the sense of the trace.
\end{defi}

\subsection{Local boundedness of solutions} 

The goal of this section is to prove $L^2\to L^\infty$ estimates for weak solutions to \eqref{eq:weak:sol}. The proof is fairly standard and employs an iterative technique based on a Caccioppoli-type inequality and the Sobolev embeddings \eqref{eq:embedding} (for example, see \cite{vasseur}). We include the proof for completeness. We start with the following Caccioppoli-type inequality.

\begin{lem}\label{lem:caccioppoli}
Let $2\le n\le d$, $a+n \in(0,2)$, $R>0$, $0\leq \e\ll1$, $p\ge(2^*)'$ and $q\ge 2$. Let $A$ be a matrix satisfying \eqref{eq:unif:ell}, $f\in L^{p,a}(B_R\setminus\Sigma_{\e})$, $F\in L^{q,a}(B_R\setminus\Sigma_{\e})^d$ and $u$ be a weak solution to \eqref{eq:weak:sol}.
Then, there exists $c>0$ depending only on $d$, $\lambda$ and $\Lambda$ such that for every $0<R_1<R_2<R$ there holds
\begin{align}\label{eq:caccioppoli}
\begin{split}
    \int_{B_{R_1}\setminus\Sigma_\e}|y|^a  |\D u|^2 dz \le c &\Big[\frac{1}{(R_2-R_1)^2}\int_{B_{R_2}\setminus\Sigma_\e}|y|^a  |u|^2 dz
    +
    \|f\|_{L^{p,a}(B_{B_{R_2}}\setminus\Sigma_\e)}
    \|u\|_{L^{p',a}(B_{B_{R_2}}\setminus\Sigma_\e)}\\
    &+ 
    \int_{B_{B_{R_2}}\setminus\Sigma_\e}|y|^a|F|^2\rchi_{\{|u|>0\}}dz
    \Big] .
\end{split}
\end{align}
Moreover, for every $b\in \R $ and for $v:=(u-b)_+=\max\{u-b,0\}$ or $v:=(u-b)_-=\max\{-u+b,0\}$ the same inequality \eqref{eq:caccioppoli} holds, with $u$ replaced by $v$.
\end{lem}

\begin{proof}

Fix $ 0  < R_1 < R_2<R$ and consider a smooth cut-off function $\eta\in C_c^\infty(B_R)$ such that
\begin{equation}\label{eq:eta:caccioppoli}
    \supp(\eta)\subset B_{R_2},\quad \eta=1\text{ on } B_{R_1},\quad 0\le \eta\le 1,\quad |\D\eta|\le \frac{c}{|R_2-R_1|},
\end{equation}
for some constant $c>0$ depending only on $d$. Let us test the equation satisfied by $u$ with $\eta^2 u$ (which is an admissible test function). Then, we obtain
\[
\int_{B_R }|y|^a \eta^2 A\D u\cdot \D u dz = \int_{B_R }|y|^a \big(
-2\eta u A\D u\cdot\D\eta+f\eta^2 u-\eta^2 F\cdot \D u -2\eta u F \cdot \D \eta 
\big)dz.
\]
By using \eqref{eq:unif:ell} and applying H\"older and Young inequalities, we obtain
\begin{align*}
    &\lambda \int_{B_R}|y|^a\eta^2|\D u|^2 dz  \le 
    2\Lambda
    \Big( \int_{B_R } |y|^a \eta^2 |\D u|^2 dz\Big)^{1/2}
    \Big( \int_{B_R }  |y|^a |u|^2 |\D \eta|^2 dz  \Big)^{1/2}\\
    &+
    \|\eta f\|_{L^{p,a}(B_{R}\setminus\Sigma_\e)}
    \|\eta u\|_{L^{p',a}(B_{R}\setminus\Sigma_\e)}
    +
    \Big( \int_{B_R } |y|^a \eta^2 |F|^2 \rchi_{\{|u|>0\}} dz  \Big)^{1/2}
    \Big( \int_{B_R }  |y|^a \eta^2 |\D u|^2  dz \Big)^{1/2}\\
    &+2
    \Big( \int_{B_R } |y|^a \eta^2 |F|^2 \rchi_{\{|u|>0\}} dz  \Big)^{1/2}
    \Big( \int_{B_R }  |y|^a |u|^2 |\D\eta|^2 dz  \Big)^{1/2}\\
    &\le 
    \frac{5}{6}\lambda \int_{B_R }|y|^a\eta^2|\D u|^2 dz+ \frac{3\Lambda^2}{\lambda} \int_{B_R }  |y|^a |u|^2 |\D \eta|^2 dz +
    \|\eta f\|_{L^{p,a}(B_{R}\setminus\Sigma_\e)}
    \| \eta u\|_{L^{p',a}(B_{R}\setminus\Sigma_\e)}\\
    &+ \Big(1+\frac{1}{2\lambda}\Big)\int_{B_R } |y|^a \eta^2|F|^2 \rchi_{\{|u|>0\}} dz + \int_{B_R }  |y|^a |u|^2 |\D\eta|^2dz.
\end{align*}
Hence, we have 
\begin{align*}
    \frac{\lambda}{6} \int_{B_{R_1} }|y|^a &|\D u|^2 dz \le \Big(1+\frac{3\Lambda^2}{\lambda}\Big) \frac{c}{{(R_2-R_1)}^2}\int_{B_{R_2} }  |y|^a |u|^2 dz \\
    &+ \|  f\|_{L^{p,a}(B_{R_2}\setminus\Sigma_\e)}
    \|  u\|_{L^{p',a}(B_{R_2}\setminus\Sigma_\e)} +  \Big(1+\frac{1}{2\lambda}\Big)\int_{B_{R_2}} |y|^a |F|^2 \rchi_{\{|u|>0\}}dz,
\end{align*}
that is, \eqref{eq:caccioppoli} holds true.

\smallskip

Finally, to prove that \eqref{eq:caccioppoli} holds true for $v:=(u-b)_+$ it is enough to choose $\eta^2 v$ as test function in \eqref{eq:weak:sol}, noting that $\D v= \D u$ on the set $\{v>0\}$ and performing similar computations as above. The case $w=(u-b)_-$ is analogous, observing that $\D w= -\D u$ on the set $\{w>0\}$.
\end{proof}

The next lemma is to establish a no-spike estimate type.

\begin{lem}\label{L:no:spike}
Let $2\le n\le d$, $a+n \in(0, 2)$, $0<r<R$, $0\leq \e\ll1$, $p>{d}/{2}$, $q>d$. Let $A$ be a matrix satisfying \eqref{eq:unif:ell}, $f\in L^{p,a}(B_R\setminus\Sigma_{\e})$, $F\in L^{q,a}(B_R\setminus\Sigma_{\e})^d$ satisfying
\[
 \|f\|_{L^{p,a}(B_{R}\setminus\Sigma_\e)}+
      \|F\|_{L^{q,a}(B_{R}\setminus\Sigma_\e)^d}\le 1.
\]
Then, there exists a constant $\delta\in(0,1)$, depending only on $d,$ $n,$ $a,$ $\lambda,$ $\Lambda,$ $p,$ $q,$ $r$ and $R$, such that if $u$ is a weak solution to \eqref{eq:weak:sol} and it satisfies
\[
\int_{B_R\setminus\Sigma_\e}|y|^a|u_+|^2dz\le \delta,
\]
then
\[
u\le1 \quad \text{ a.e. in } B_r\setminus\Sigma_\e.
\]
Conversely, if
 \[
\int_{B_R\setminus\Sigma_\e}|y|^a|u_-|^2dz\le \delta,
\]
then
\[
u\ge -1 \quad \text{ a.e. in } B_r\setminus\Sigma_\e.
\]
\end{lem}

\begin{proof}
Let us assume that $d\ge3$ (the case $d=2$ is similar and relies only on the Sobolev-type inequality \eqref{eq:embedding}) and consider the positive part $u_+$ (the case $u_-$ is exactly the same).

For every $j\in\N$, set
\[
C_j:=1-2^{-j},\quad r_j:=(R-r)2^{-j}+r,\quad D_j:=B_{r_j}\setminus\Sigma_\e,
\]
noting that $C_0=0$, $r_0=R$, $C_j\uparrow 1$, $r_j\downarrow r$, $D_j\supset D_{j+1}$ and $r_j-r_{j+1}=(R-r)2^{-(j+2)} $. We define
\[
V_j:=(u-C_j)_+, \quad E_j:=\int_{D_j}|y|^a V_j^2dz,
\]
which satisfy, for every $j\in\N$, $E_{j+1}\le E_j\le E_0\le \delta$ by assumption. 

Applying the Sobolev inequality \eqref{eq:embedding} and the Caccioppoli-type inequality \eqref{eq:caccioppoli} to $V_{j+1}$, with $R_1=r_{j+1}$, $R_2=r_j$, thanks to the assumption $\|f\|_{L^{p,a}(B_R\setminus\Sigma_\e)}\le 1$, we get
\begin{align}\label{eq:no:spike:1}
\begin{split}
    &\Big(\int_{D_{j+1}}|y|^a |V_{j+1}|^{2^*} dz\Big)^{{2}/{2^*}} \le c   
    \int_{D_{j+1}}|y|^a |V_{j+1}|^{2} dz   \\
  & \le c\Big[
2^{2j}\int_{D_j}|y|^a |V_{j+1}|^2dz+\|V_{j+1}\|_{L^{p',a}(D_{j})}
+
\int_{D_j}|y|^a |F|^2\rchi_{\{V_{j+1}>0\}}dz
\Big].
\end{split}
\end{align}
Next, by using  H\"older inequality in the definition of $E_{j+1}$ it follows that
\begin{equation}\label{eq:no:spike:2}
        E_{j+1}=\int_{D_{j+1}}|y|^a V_{j+1}^2dz\le \Big(
        \int_{D_{j+1}}|y|^a |V_{j+1}|^{2^*} dz
        \Big)^{{2}/{2^*}}
        \Big(
        \int_{D_{j+1}}|y|^a \rchi_{\{V_{j+1}>0\}} dz
        \Big)^{(2^*-2)/{2^*}}.
\end{equation}
Again, we use H\"older inequality, noting that \( V_{j+1} \leq V_j \), to estimate 
\begin{align}\label{eq:no:spike:3}
\begin{split}
    \|V_{j+1}\|_{L^{p',a}(D_{j})}&\le  \Big(
    \int_{D_j}|y|^a V_{j+1}^2dz
    \Big) ^{1/2}
    \Big(
    \int_{D_j}|y|^a \rchi_{\{V_{j+1}>0\}}dz
    \Big)^{(p-2)/{2p}}\\
    &\le  E_j^{1/2} \Big(
    \int_{D_j}|y|^a \rchi_{\{V_{j+1}>0\}}dz
    \Big)^{(p-2)/{2p}},
    \end{split}
\end{align}
and, by also using $\|F\|_{L^{q,a}(B_R\setminus\Sigma_\e)^d}\le 1$, we obtain
\begin{align}\label{eq:no:spike:4}
\begin{split}
   \int_{D_j}|y|^a |F|^2\rchi_{\{V_{j+1}>0\}}dz&\le 
   \Big(
    \int_{D_j}|y|^a |F|^qdz
    \Big) ^{{2}/{q}}
    \Big(
    \int_{D_j}|y|^a \rchi_{\{V_{j+1}>0\}}dz
    \Big)^{{(q-2})/{q}}\\
    &\le \Big(
    \int_{D_j}|y|^a \rchi_{\{V_{j+1}>0\}}dz
    \Big)^{{q-2}/{q}}.
    \end{split}
\end{align}
Moreover, one has that
\[
\{V_{j+1}>0\}=\{u-C_{j+1}>0\}=\{u-C_j>2^{-(j+1)}\}=\{V_j>2^{-(j+1)}\},
\]
hence,
\begin{equation}\label{eq:no:spike:5}
   \int_{D_j}|y|^a \rchi_{\{V_{j+1}>0\}}dz = \int_{D_j}|y|^a \rchi_{\{V_{j}^2>2^{-2(j+1)}\}} dz\le 2^{2(j+1)}\int_{D_j}|y|^a V_j^2 dz = 2^{2(j+1)} E_j. 
\end{equation}
Putting together \eqref{eq:no:spike:1}, \eqref{eq:no:spike:2}, \eqref{eq:no:spike:3}, \eqref{eq:no:spike:4} and \eqref{eq:no:spike:5} we obtain
\[
E_{j+1}\le C^{j+1}(E_j+E_j^{1-\frac{1}{p}}+E_j^{1-\frac{2}{q}} )E_j^\frac{2^*-2}{2^*},
\]
where $\frac{2^*-2}{2^*}-\frac{1}{p}>0$ and $\frac{2^*-2}{2^*}-\frac{2}{p}>0$ by the assumptions $p>d/2$ and $q>d$. By denoting with $\bar{\gamma}>0$ the minimum of these two numbers we have that 
\[
\begin{cases}
    E_{j+1}\le C^{j+1}E_j^{1+\bar \gamma},\\
    E_0\le \delta,
\end{cases}
\]
which implies
\[
E_j\le C^{\sum_{i=0}^ji(1+\gamma)^{j-i}} E_0^{(1+\gamma)^j}\le (C\delta)^{(1+\gamma)^j}.
\]
Finally, by choosing $\delta$ such that $C\delta<1$, and taking the limit as $j\to\infty$ we obtain that $E_j\to0 $, that is, $\int_{B_r\setminus\Sigma_\e}|y|^a (u-1)^2_+=0$, which yields $u\le 1$ a.e. in $B_r\setminus\Sigma_\e$.
\end{proof}

Finally, the next lemma states the $L^\infty_\loc$ boundedness of weak solutions to \eqref{eq:weak:sol}.

\begin{lem}\label{lem:moser}
Let $2\le n\le d$, $a+n \in(0, 2)$, $R>0$, $0\leq \e\ll1$, $p>{d}/{2}$, $q>d$. Let $A$ be a matrix satisfying \eqref{eq:unif:ell}, $f\in L^{p,a}(B_R\setminus\Sigma_{\e})$, $F\in L^{q,a}(B_R\setminus\Sigma_{\e})^d$ and let $u$ be a weak solution to \eqref{eq:weak:sol}. Then, for every $r\in(0,R)$, there exists $c > 0$ depending only on $d$, $n$, $a$, $\lambda$, $\Lambda$, $p$, $q$ and $r$ such that
\begin{equation}\label{eq:moser}
\|u\|_{L^\infty(B_{r}\setminus \Sigma_\e)}\le c\big(
\|u\|_{L^{2,a}(B_{R}\setminus \Sigma_\e)}
+\|f\|_{L^{p,a}(B_{R}\setminus \Sigma_\e)}
+\|F\|_{L^{q,a}(B_{R}\setminus \Sigma_\e)^d}
\big).
\end{equation}
\end{lem}

\begin{proof}
   Let us define
   \[
   v:=\theta u,\quad\theta:=\frac{\sqrt{\delta}}{
   \|u\|_{L^{2,a}(B_R\setminus\Sigma_\e)}+
   \|f\|_{L^{p,a}(B_R\setminus\Sigma_\e)}+
   \|F\|_{L^{q,a}(B_R\setminus\Sigma_\e)^d}},
   \]
   where $\delta>0$ is as in Lemma \ref{L:no:spike}.
   One has that $v$ satisfies the hypothesis of Lemma \ref{L:no:spike}, hence, $v_+\le 1$ in $B_r\setminus\Sigma_\e$, which implies
   \[
   \|u_+\|_{L^\infty(B_r\setminus\Sigma_\e)}\le \frac{1}{\sqrt \delta}(\|u\|_{L^{2,a}(B_R\setminus\Sigma_\e)}+
   \|f\|_{L^{p,a}(B_R\setminus\Sigma_\e)}+
   \|F\|_{L^{q,a}(B_R\setminus\Sigma_\e)^d}).
   \]
Repeating the same argument with $v_-$ one has that
   \[
   \|u_-\|_{L^\infty(B_r\setminus\Sigma_\e)}\le \frac{1}{\sqrt \delta}(\|u\|_{L^{2,a}(B_R\setminus\Sigma_\e)}+
   \|f\|_{L^{p,a}(B_R\setminus\Sigma_\e)}+
   \|F\|_{L^{q,a}(B_R\setminus\Sigma_\e)^d}).
   \]
Hence, we have that \eqref{eq:moser} holds true, by choosing $C=2/\sqrt \delta$. The proof is complete.
\end{proof}

\subsection{Approximation result}\label{s:approximation}

In the spirit of \cite{SirTerVit21a}*{Lemma 2.12, Lemma 2.15} and \cite{AFV24}*{Lemma 4.2}, the goal of this section is to provide an approximation result, which allows us to construct a family of solutions to \eqref{eq:weak:sol} in perforated domains ($0<\e\ll 1$), which converges in a suitable sense to weak solutions of \eqref{eq:weak:sol}, when $\e=0$.

\begin{lem}\label{lem:approximation}

Let $2\le n\le d$, $a+n\in(0,2)$, $R>0$. Let $A$ be a matrix satisfying \eqref{eq:unif:ell}, $f\in L^{2,a}(B_R)$, $F\in L^{2,a}(B_R)^d$ and let $u$ be a weak solution to \eqref{eq:weak:sol} with $\e=0$.

Then, for every $r\in(0,R)$, there exists a family $\{u_\e\}_{0<\e\ll 1}$, such that $u_\e$ are weak solutions to
\begin{equation}\label{eq:lemm:approximation:e}
\begin{cases}
-\div(|y|^a A \nabla u_\e)= |y|^a f+\div(|y|^aF), & \text{in }B_{r}\setminus  \Sigma_{\e}, \\
u=0, & \text{on } \partial \Sigma_\e \cap B_{r},
\end{cases}
\end{equation}
satisfying
\begin{align}\label{eq:stima:rhs}
\begin{aligned}
&\|u_\e\|_{H^{1,a}(B_{r}\setminus\Sigma_{\e})}\le c \big(\|u\|_{H^{1,a}(B_R)}+\|f\|_{L^{2,a}(B_R)}
+\|F\|_{L^{2,a}(B_R)^d}\big)
,
\end{aligned}
\end{align}
for some constant $c>0$ depending only on $d$, $n$, $a$, $\lambda$, $\Lambda$, $R$, $r$ and, up to consider the trivial extension of $u_{\e}$ in the whole $B_r$ (see Remark \ref{rem:inclusion}), there exists a sequence $\e_k\to0$ such that
\begin{equation}\label{eq:convergence:approximation}
 u_{\e_k}\to u \text{ in } H^{1,a}(B_{r}).
\end{equation}

\end{lem}

\begin{proof}
Let us fix $r\in (0,R)$ and consider a cut-off function $\xi\in C_c^\infty(B_{R})$ such that 
\[\xi=1 \text{ in } B_{r},\quad \supp(\xi)\subset B_{\frac{R+r}{2}},\quad 0\le\xi\le1,\quad|\D\xi|\le c_0,\]
for some $c_0>0$ depending only on $d$, $R$ and $r$, and define $\tilde{u}=\xi u \in H^{1,a}_0(B_R) $.

Fixed $\phi\in C_c^\infty(B_R\setminus\Sigma_0)$, and using the equation \eqref{eq:weak:sol:integral} satisfied by $u$, we get
\begin{align*}
&\int_{B_R}|y|^a   A\D \tilde{u}\cdot\D\phi  dz    = \int_{B_R}|y|^a  \Big( \xi A\D {u} \cdot\D\phi   u   A\D \xi \cdot\D\phi   \Big)    dz   \\
&= 
\int_{B_R}|y|^a  \Big( A\D {u} \cdot \D(\phi\xi) - 
 \phi A\D {u} \cdot \D \xi +
 u   A\D \xi \cdot\D\phi \Big)   dz   \\
&= 
\int_{B_R}|y|^a\Big( 
f\phi\xi - F\cdot \D (\phi\xi)
- 
 \phi A\D {u} \cdot \D \xi +
 u   A\D \xi \cdot\D\phi \Big)   dz   \\
 &= 
 \int_{B_R}|y|^a \Big( 
 f\phi\xi
 - \xi F\cdot \D \phi
 - \phi F\cdot \D \xi
 - 
 \phi A\D {u} \cdot \D \xi +
 u   A\D \xi \cdot\D\phi \Big)   dz   ,
\end{align*}
that is, $\tilde{u}$ is a weak solution to 
\begin{equation}\label{eq:cut:off:approximation}
\begin{cases}
-\div(|y|^a A \D\tilde{u})=|y|^a \tilde{f}+ \div(|y|^a  \tilde{F}),& \text{ in }B_R\setminus\Sigma_0,\\
u=0, &  \text{ on } \partial B_R\setminus\Sigma_0,\\
u=0, &  \text{ on } \Sigma_0 \cap B_R,
\end{cases}
\end{equation}
where we have set
\[
\tilde{f}= f\xi-F\cdot\D\xi-A\D u\cdot \D\xi,
\qquad
\tilde{F}= F\xi-uA\D\xi.
\]
We estimate the right hand side in \eqref{eq:cut:off:approximation} in the following way
\begin{align}\label{eq:rhs:f:cutoff}
\begin{aligned}
&\int_{B_R }|y|^a \tilde{f}^2 dz \le 
2 \int_{B_R }|y|^a
\Big(
(\xi f)^2+(F\cdot\D\xi)^2+(A\D u\cdot\D\xi)^2
\Big) dz  \\
&\le 2 \int_{B_R }|y|^a
\Big(
f^2+c_0^2|F|^2+|A\D\xi|^2 |\D u|^2
\Big)   dz\\
&\le c
\int_{B_R }|y|^a \Big(f^2 + |F|^2   + |\D u|^2   \Big)dz 
,
\end{aligned}
\end{align}
for some $c>0$ depending only on $d$, $\Lambda$, $R$ and $r$. By performing similar computations, we get 
\begin{align}\label{eq:rhs:F:cutoff}
\begin{aligned}
&\int_{B_R }|y|^a |\tilde{F}|^2 dz
&\le c
\int_{B_R}|y|^a \big(|F|^2  
+ u^2  
\Big)dz,
\end{aligned}
\end{align}
for some $c>0$ depending only on $d$, $\Lambda$, $R$ and $r$

\medskip

Fixed $0<\e_0\ll 1$, for every $0<\e<\e_0$, recalling Remark \ref{rem:unicity}, let $u_\e$ be the unique weak solution to 
\begin{equation}\label{eq:approximation:u:e}
\begin{cases}
-\div(|y|^a A \nabla u_\e)= |y|^a \tilde{f}+\div(|y|^a \tilde{F}), & \text{in }B_R\setminus \Sigma_{\e},\\
u_\e=0, & \text{on }\partial B_R \setminus \Sigma_\e,\\
u_\e=0, & \text{on } \Sigma_\e\cap B_R.
\end{cases}
\end{equation}
By using $u_\e \in H^{1,a}_0(B_R\setminus\Sigma_\e)$ as test function in \eqref{eq:approximation:u:e}, up to consider the trivial extension in the whole $B_r$ (see Remark \ref{rem:inclusion}), combined with \eqref{eq:unif:ell}, Poincaré inequality \eqref{eq:poincaré} and H\"older inequality, we get
\begin{align*}
&\lambda \int_{B_R}   |y|^a |\D u_\e|^2 dz\le 
\int_{B_R}   |y|^a A \D  u_\e  \cdot \D u_\e dz = 
\int_{B_R}   |y|^a \Big(\tilde{f} u_\e +\tilde F \cdot \D u_\e\Big) dz \\
&\le\Big(\int_{B_R}   |y|^a |\tilde f|^2 dz\Big)^{1/2}
\Big(\int_{B_R}   |y|^a |u_\e|^2 dz\Big)^{1/2}+
\Big(\int_{B_R}   |y|^a |\tilde F|^2 dz\Big)^{1/2}
\Big(\int_{B_R}   |y|^a |\D u_\e|^2 dz\Big)^{1/2}\\
&\le c \Big(\int_{B_R}   |y|^a |\D u_\e|^2 dz\Big)^{1/2} \Big(
\|\tilde{f}\|_{L^{2,a}({B_R})}+ \|\tilde{F}\|_{L^{2,a}({B_R})^d}
 \Big),
\end{align*}
and then, by using \eqref{eq:rhs:f:cutoff} and \eqref{eq:rhs:F:cutoff}, we have that there exists a constant $c>0$ depending only on $d$, $n$, $a$, $\lambda$ and $\Lambda$ such that
\begin{align}\label{eq:ssss:stima}
\begin{aligned}
\|u_\e\|_{H^{1,a}(B_R\setminus\Sigma_\e)}\le c \big(\|f\|_{L^{2,a}(B_R)}
+\|F\|_{L^{2,a}(B_R)^d}
+\|u\|_{H^{1,a}( B_R)}\big).
\end{aligned}
\end{align}
So, we get that $\{u_\e\}\subset H^{1,a}_0(B_R\setminus\Sigma_\e) \subset H^{1,a}_0(B_R)$ is uniformly bounded. Hence, there exists  $\bar{u} \in H^{1,a}_0(B_R)$ and a sequence $\e_k\to0$, such that
\begin{equation}\label{eq:weak:D:u:eps}
u_{\e_k}\rightharpoonup \bar{u},\quad \text{weakly in } H^{1,a}(B_R).
\end{equation}

\medskip

Next, we prove that $\bar{u}=\tilde{u}$.  
Let $\phi\in C_c^\infty(B_R\setminus\Sigma_0)$ be a test function in the equation \eqref{eq:approximation:u:e} satisfied by $u_\e$. Then, $\supp({\phi})\subset B_R\setminus \Sigma_\e$ for every $\e$ small enough. By using \eqref{eq:weak:D:u:eps}, we have
\[
	\int_{\supp({\phi})}|y|^a \tilde{f}\phi-\tilde{F}\cdot \D \phi=
\int_{\supp({\phi})}|y|^a A\D u_{\e_k}\cdot \D\phi\to\int_{\supp({\phi})}|y|^a A\D \bar{u}\cdot \D\phi,\quad \text{as }\e_k\to0,
\]
so, $\bar{u}$ is a weak solution to 
\[
\begin{cases}
-\div(|y|^a A \nabla \bar{u})= |y|^a \tilde{f}+\div(|y|^a \tilde{F}), & \text{in }B_R\setminus \Sigma_{0},\\
\bar{u}=0, & \text{on }\partial B_R \setminus \Sigma_0,\\
\bar{u}=0, & \text{on } \Sigma_0\cap B_R.
\end{cases}
\]
By uniqueness of weak solution to \eqref{eq:cut:off:approximation} (see  Remark \ref{rem:unicity}), we get that $\bar{u}=\tilde{u}$ in $H^{1,a}_0(B_R).$

\medskip

Finally, we prove that $u_{\e_k}\to \tilde{u}$ strongly in $H^{1,a}(B_R)$. By testing \eqref{eq:cut:off:approximation} with $\tilde u$, we get
\begin{equation}\label{eq:test:tilde:u}
\int_{B_R}|y|^a A\D\tilde u\cdot \D\tilde{u}dz = \int_{B_R}|y|^a \Big(\tilde{f} \tilde{u}-\tilde{F}\cdot \D \tilde{u}\Big)dz,
\end{equation}
and, by testing \eqref{eq:approximation:u:e} with $u_\e$ combined with \eqref{eq:weak:D:u:eps}, we have
\begin{equation}\label{eq:test:tilde:u:eps}
\int_{B_R}|y|^a A\D  u_\e\cdot \D{u_\e}dz = \int_{B_R}|y|^a \Big(\tilde{f} \tilde{u}_\e-\tilde{F}\cdot \D \tilde{u}_\e\Big)dz\to \int_{B_R}|y|^a \Big(\tilde{f} \tilde{u}-\tilde{F}\cdot \D \tilde{u}\Big)dz,
\end{equation}
along a subsequence $\e_k\to0$.
Putting together \eqref{eq:test:tilde:u} and \eqref{eq:test:tilde:u:eps} we obtain
\[
\lim_{\e_k\to0}\int_{B_R}|y|^a A\D  u_{\e_k}\cdot \D{u_{\e_k}}dz = \int_{B_R}|y|^a A\D\tilde u\cdot \D\tilde{u}dz.
\]
Since $A$ satisfies \eqref{eq:unif:ell}, one has that $\|\D u_{\e_k}\|_{L^{2,a}(B_R)}\to \|\D \tilde u\|_{L^{2,a}(B_R)}$. This, combined with \eqref{eq:weak:D:u:eps}, allows us to assert that 
\begin{equation}\label{eq:strong:D:u:eps}
u_{\e_k}\to \tilde{u},\quad \text{strongly in } H^{1,a}(B_R).
\end{equation}

Finally, since $\tilde{u}=u$, $\tilde{f}=f$, $\tilde{F}=F$ in $B_{r}$, we have that $u_\e$ is a weak solution to \eqref{eq:lemm:approximation:e} in $B_r$ and, by using \eqref{eq:ssss:stima} and \eqref{eq:strong:D:u:eps}, our statement follows.
\end{proof}

\section{Liouville theorems}\label{section:liouville}

The goal of this section is to prove the Liouville type Theorem \ref{teo:liouville} for homogeneous entire solutions to \eqref{eq:entire:solution} with constant coefficients. The proof is based on a spectral trace inequality which is stable with respect to $\e$, using an argument similar to \cite{SirTerVit21a}*{Theorem 3.4}.
We start with a couple of results which are crucial to treat the case $n=d$. 

\begin{lem}\label{lem:trace}
Let $ n= d$, $a+n\in(0,2)$, $0\le\e\ll 1$ and $r>0$. Let $A\in \R^{n,n}$ be a positive definite symmetric matrix and define $\Omega_r := \{ A^{-1} y\cdot y < r^2\}$. Then,
\begin{equation}\label{eq:spectral:0}
\int_{\Omega_r}  |y|^a  A \D v \cdot \D v   dz  \geq (2-a-n)\int_{ \partial \Omega_r} |y|^a v^2 d\sigma,
\end{equation}
for every $v \in C_c^\infty(\overline{\Omega}_r \setminus \Sigma_\e)$.

\end{lem}

\begin{proof}
We provide the result for $r=1$. The case for generic $r>0$ follows by a scaling argument. 

Since $A$ is a positive definite symmetric matrix, it is well defined the square root $A^{1/2}$, which is a positive definite symmetric matrix too.
The homogeneous function
\[
\bar{u}(y):=|A^{-1/2} y|^{2-a-n},
\]
is solution (in a point-wise sense) to 
\begin{equation}\label{eq:special2}
-\div(|y|^a A  \nabla \bar{u})= 0, \quad \text{in } {\R^n}\setminus \Sigma_{0},
\end{equation}
and satisfies
\begin{equation}\label{eq:grad:special}
\nabla \bar{u}(y)=(2-a-n)|A^{-1/2} y |^{-a-n} A^{-1} y,\qquad \nabla \bar{u}(y)\cdot y= (2-a-n)\bar{u}(y).
\end{equation}
Indeed, equations \eqref{eq:special2} and \eqref{eq:grad:special}  are verified by a straightforward computation.

Fix $v\in C_c^\infty(\overline{\Omega}_1\setminus \Sigma_\e)$. Then,
\begin{equation}\label{eq:first:spectral}
\int_{\Omega_1}|y|^a A \nabla \bar{u}\cdot\nabla\Big(\frac{v^2}{\bar{u}}\Big)dy 
=
\int_{\Omega_1}|y|^a  \Big( A \nabla v\cdot \nabla v
-
 \Big| A^{1/2} \D v- \frac{v}{\bar{u}} A^{1/2}\D \bar{u}\Big|^2 \Big)  dy
\le 
\int_{\Omega_1}|y|^a  A \nabla v\cdot \nabla vdy,
\end{equation}
On the other hand, by using the divergence theorem, \eqref{eq:special2} and \eqref{eq:grad:special}, we have
\begin{equation}\label{eq:second:spectral}
\int_{\Omega_1}|y|^a A \nabla \bar{u}\cdot\nabla\Big(\frac{v^2}{\bar{u}}\Big) dy
=
\int_{\partial \Omega_1}|y|^a \frac{v^2}{\bar{u}} A \nabla \bar{u}\cdot \nu d\sigma
=
(2-a-n)\int_{\partial \Omega_1}|y|^a  v^2  d\sigma ,
\end{equation}
so, putting together \eqref{eq:first:spectral} and \eqref{eq:second:spectral}, our statement follows.
\end{proof}

\begin{lem}\label{lem:almgren}
Let $ n = d$, $a+n \in (0,2)$, $R>0$ and $0\le\e\ll 1$. Let $A\in \R^{n,n}$ be a positive definite diagonal matrix and define $\Omega_r := \{ A^{-1} y\cdot y < r^2\}$, for every $r>0$ such that $\partial \Omega_r \subset B_R\setminus\Sigma_\e$. Let $u$ be a weak solution to
\begin{equation}\label{eq:lem:alm}
\begin{cases}
-\div(|y|^a  A\nabla {u})= 0, & \text{in }B_R\setminus \Sigma_{\e},\\
{u}=0, & \text{on } \partial\Sigma_\e \cap B_R,
\end{cases}
\end{equation}
Up to consider the trivial extension of $u$ in the whole $B_R$ (see Remark \ref{rem:inclusion}), let us define
\begin{align*}
&E(u,r):=   \frac{1}{r^{n+a-2}}\int_{\Omega_r}|y|^a A\nabla u\cdot \D u  dy, \\
&H(u,r):=  \frac{1}{r^{n+a-1}} \int_{\partial \Omega_r }|y|^a u^2 d\sigma.
\end{align*}
Then,
\begin{equation}
\partial_r H(v,r)=\frac{2}{r}E(v,r),\quad  \text{ for every } r\in(0,R).
\end{equation}
\end{lem}

\begin{proof}

When \(\varepsilon > 0\), classical regularity theory ensures that the function \(u\) is smooth in \(\overline{\Omega_r \setminus \Sigma_\varepsilon}\). Consequently, the result immediately follows trough explicit computations.

\smallskip

When $\e=0$, we proceed by an approximation argument. 
Fixed $0<\delta\ll 1$, by using the approximation Lemma \ref{lem:approximation}, we find a family $\{u_\e\}_{0<\e\ll 1}$ of solutions to \eqref{eq:lem:alm} in ${B}_{R-\delta}\setminus\Sigma_\e$ such that 
$u_{\e_k}\to u$ in $H^{1,a}({B}_{R-\delta})$ along a sequence ${\e_k}\to0$ and, by applying the trace Poincaré inequality \eqref{eq:poincaré:trace} (which also holds in $\Omega_r$) we get that $v_{\e_k}\to v$ in $L^{2,a}(\partial \Omega_r)$. Hence, we have that
\begin{align}\label{eq:almgren:limit}
\begin{split}
&\int_{\Omega_r}|y|^a A\D u_{\e_k} \cdot \D u_{\e_k} dy \to \int_{\Omega_r}|y|^a A\D u \cdot \D u dy ,
\\
&\int_{\partial \Omega_r}|y|^a  u_{\e_k}^2d\sigma \to \int_{\partial \Omega_r}|y|^a u^2d\sigma.
\end{split}
\end{align}
By utilizing the result obtained in the case $\e>0$ one finds that
\begin{equation}\label{eq:almgren:approximation}
\partial_r H(u_{\e_k},r)=\frac{2}{r}E(u_{\e_k},r).
\end{equation} 
By applying \eqref{eq:almgren:limit}, we can take the limit as $\e_k\to 0$ in \eqref{eq:almgren:approximation} to obtain $\partial_r H(u,r)=\frac{2}{r}E(u,r)$.
\end{proof}

The following lemma allows us to handle the unweighted variables
$x$. Its proof relies on the method of difference quotients and an iterative application of the Caccioppoli-type inequality \eqref{eq:caccioppoli}. For a detailed proof, see \cite{TerTorVit22}*{Corollary 4.2, Lemma 4.3} in a quite similar context.

\begin{lem}\label{L:dx}

Let $2\le n< d$, $a+n\in(0,2)$, $\e\ge0$. Let $A$ be a constant symmetric matrix satisfying \eqref{eq:unif:ell} and $u$ be an entire solution to \eqref{eq:entire:solution}. Then, the following holds true.
\begin{itemize}

\item[i)] For every $i=1,\dots,d-n$, the function $\partial_{x_i} u $ is an entire solution to the same problem.

\item[i)]  If $u$ satisfies the growth condition \eqref{eq:growth} for some $\gamma>0$, then $u$ must be a polynomial in the variable $x$ of degree almost $\lfloor \gamma \rfloor$.

\end{itemize}

\end{lem}

\begin{proof}[Proof of Theorem \ref{teo:liouville}.]

Let $u$ be an entire solution to \eqref{eq:entire:solution} and let us suppose that $n=d$ and so $z=y\in\R^n$. 

By contradiction let us suppose that $u\not\equiv0$. Let $r_0>0$ such that $\Sigma_\e \subset \Omega_r := \{A^{-1} y \cdot y \le r^2\}$ for every $r\ge r_0$ and define
\[
E(u,r)=\frac{1}{r^{n+a-2}}\int_{\Omega_r }|y|^a A \D u\cdot \D u dy,
\]
\[
H(u,r)=\frac{1}{r^{n+a-1}}\int_{\partial \Omega_r}|y|^a u^2 d\sigma.
\]
By applying Lemma \ref{lem:trace} and Lemma \ref{lem:almgren}, we get 
\[\partial_r H(u,r)=\frac{2}{r}E(u,r)\ge \frac{2(2-a-n)}{r}H(u,r),$$
which implies
$$H(u,r)\ge H(u,r_0)r^{2(2-a-n)},\quad \text{ for every } r>r_0,\]
by Gronwall's inequality. On the other hand, since $A$ satisfies \eqref{eq:unif:ell}, the growth condition \eqref{eq:growth} implies
\[H(u,r)\le c (1+r^{2\gamma}).\]
Combining these two inequalities we get
\[
H(u,r_0)\le cr^{2(\gamma-(2-a-n))}.
\]
Taking the limit as $r\to \infty$ and using $\gamma<2-a-n$ we get $H(u,r_0)=0$. Since $r_0>0$ is arbitrary, we deduce that $u\equiv 0$ in $\R^d\setminus \Omega_{r_0}$. Moreover, since $u$ is a solution to \eqref{eq:entire:solution} and satisfies $u=0$ on $\partial (\Omega_{r_0}\setminus \Sigma_\e)$, we apply the existence and uniqueness result (see Remark \ref{rem:unicity}) to conclude that $u \equiv 0$ in  $\Omega_{r_0}\setminus \Sigma_\e$. Therefore, $u\equiv 0 $ in $\R^d\setminus\Sigma_\e$, which leads to a contradiction.

\medskip

Let us consider the case $n<d$.
By Lemma \ref{L:dx}, one has that $u$ is polynomial in the variable $x$.
Hence, if $\gamma\in(0,1)$, the function $u$ must be constant in $x$, so $u(x,y)=u(y)$ and our statement follows by using the result obtained in the case $n=d$.
If $\gamma\in[1,2)$, we have that $u$ must be linear in $x$, that is
\[u(x,y)= u_i(y)+\sum_{i=1}^{d-n}x_iu_i(y),\]
for some unknown functions $u_i(y)$. 
First, 
\[|u_0(y)|=|u(0,y)|\le c(1+|y|^\gamma).\]
On the other hand,
\[|u(e_{x_i},y)|=|u_i(y)+u_0(y)|\le c(1+|y|^\gamma),\]
and so
\[ |u_i(y)| \le |u_0(y)|+ c(1+|y|^\gamma)\le  c(1+|y|^\gamma). \]
Hence, every $u_i$ satisfies the growth condition \eqref{eq:growth} for every $i=0,\dots,d-n$. 

Next, for every $i=1,\dots,d-n$, by applying Lemma \ref{L:dx} we have that $\partial_{x_i}u(x,y) = u_i(y)$ is an entire solution to \eqref{eq:entire:solution} and satisfies \eqref{eq:growth}. Then, the result obtained in the case $d=n$ allows us to conclude that $u_i=0$ for every $i=1,\dots,d-n$ and so $u(x,y)=u_0(y)$. Hence, using again the case $d=n$, we have that $u$ must be zero and our statement follows.
\end{proof}

\section{H\"older estimates for weak solutions}\label{section:holder}

The goal of this section is to prove Theorem \ref{teo:0}, which we obtain as a by-product of $\e$-uniform H\"older estimates for solutions in perforated domains and the approximation Lemma \ref{lem:approximation}.

\begin{teo}\label{teo:0alpha:e}
Let $2\le n\le d$, $a+n\in(0,2)$, $p>{d}/{2}$, $q>d$ and $\alpha$ satisfying \eqref{eq:alpha:0}.
Let $A$ be a continuous matrix satisfying \eqref{eq:unif:ell} and $\|A\|_{C^0({B_1})}\le L$, $f\in L^{p,a}(B_{1})$ and $F\in L^{q,a}(B_{1})^d$. For $0<\e\ll1$, let $\{u_\e\}$ be a family of solutions to 
\begin{equation}\label{eq:sol:c0}
\begin{cases}
-\div(|y|^a A \nabla u_\e)= |y|^a f+\div(|y|^aF), & \text{in }B_1\setminus \Sigma_{\e},\\
u_\e=0, & \text{on } \partial \Sigma_\e \cap B_1.
\end{cases}
\end{equation}
Then, there exists a constant $c>0$, depending only on $d$, $n$, $a$, $\lambda$, $\Lambda$, $p$, $q$, $\alpha$ and $L$ such that
\begin{equation}\label{eq:stima:c0:e}
\|u_\e\|_{C^{0,\alpha}(B_{1/2}\setminus\Sigma_\e)}\le c\big(
\|u_\e\|_{L^{2,a}(B_{1}\setminus\Sigma_\e)}
+\|f\|_{L^{p,a}(B_{1})}
+\|F\|_{L^{q,a}(B_{1})^d}
\big).
\end{equation}

\end{teo}

\begin{proof}
By classical regularity theory, we know that solutions to \eqref{eq:sol:c0} are $C^{0,\alpha}(B_{1/2}\setminus\Sigma_\e)$ and that \eqref{eq:stima:c0:e} holds with a constant $c>0$ that may also depend on $\e$. Our goal is to show that it is possible to provide a constant $c>0$ that is uniform in $\e$.

Without loss of generality, we can assume that 
\[
\|u_\e\|_{L^{2,a}(B_{1}\setminus\Sigma_\e)}
+\|f\|_{L^{p,a}(B_{1})}
+\|F\|_{L^{q,a}(B_{1})^d}\le c,
\]
for some $c>0$, which not depends on $\e$. Moreover, by using the local uniform bound of weak solutions in (see Lemma \ref{lem:moser}), it follows that 
\begin{equation}\label{eq:c0:l:infty}
\|u_\e\|_{L^{\infty}(B_{3/4}\setminus\Sigma_\e)}\le c,
\end{equation}
for some $c>0$, which not depends on $\e$.

\medskip
\noindent
\emph{Step 1. Contradiction argument and blow-up sequences.} 
By contradiction let us suppose that there exist $p>{d}/{2}$, $q>d$, $\alpha$ satisfying \eqref{eq:alpha:0}, $\{u_k\}_k:=\{u_{\e_k}\}_k$ as $\e_k\to0$ such that 
\begin{equation}\label{eq:sol:c0:k}
\begin{cases}
-\div(|y|^a A \nabla u_k)= |y|^a f+\div(|y|^aF), & \text{in }B_1\setminus \Sigma_{\e_k},\\
u_k=0, & \text{on } \partial \Sigma_{\e_k}\cap B_1,
\end{cases}
\end{equation}
and 
\begin{equation*}
\|u_k\|_{C^{0,\alpha}(B_{1/2}\setminus\Sigma_{\e_k})}\to \infty. 
\end{equation*}
Let us fix a smooth cut-off function $\eta \in C_c^\infty(B_1)$ such that
\[
\supp(\phi)\subset B_{3/4},
\quad 
\eta=1 \text{ in }B_{1/2},
\quad
0\le\eta\le1.
\]
By \eqref{eq:c0:l:infty}, one has that 
\begin{equation*}
L_k:=[\eta u_k]_{C^{0,\alpha}(B_{1}\setminus\Sigma_{\e_k})} \to \infty.
\end{equation*}
By definition of H\"older seminorm, take two sequences of points $z_k=(x_k,y_k),\hat{z}_k=(\hat x_k,\hat y_k)\in B_{1}\setminus\Sigma_{\e_k}$ such that
\begin{equation}\label{eq:c0:alpha:points}
\frac{|(\eta u_k)(z_k)-(\eta u_k)(\hat z_k)|}{|z_k-\hat z_k|^\alpha}\ge\frac{L_k}{2},
\end{equation}
define $r_k:=|z_k-\hat z_k|$ and observe that at least one of $z_k$ or $\hat{z}_k$ belongs to $B_{3/4} \setminus \Sigma_{\e_k}$.
First, by using the local uniform bound of weak solutions \eqref{eq:c0:l:infty}, we have that $r_k\to0$, in fact
\[
L_k\le\frac{4\|\eta u_k\|_{L^\infty(B_{3/4}\setminus\Sigma_{\e_k})}}{r_k^\alpha}\le \frac{c}{ r_k^\alpha},
\]
which implies
\[
r_k\le \frac{c^{1/\alpha}}{L_k^{1/\alpha}}\to 0,\quad\text{ as } k\to \infty.
\]
From now on we distinguish three cases.
 
\smallskip

\begin{itemize}[left=0pt]
    \item \textbf{Case 1:} $ 
    \displaystyle{
    \frac{|y_k|}{r_k}\to \infty, \quad \frac{|y_k|-\e_k}{r_k}\to \infty,}$
    
        \item \textbf{Case 2:} $ 
    \displaystyle{
    \frac{|y_k|}{r_k}\to \infty, \quad \frac{|y_k|-\e_k}{r_k}\le c},
     $
    
\item \textbf{Case 3:}
    $ 
    \displaystyle{
    \frac{|y_k|}{r_k}\le c, }$
    
\end{itemize}
for some constant $c>0$ which not depends on $k$.
Let $z_k^0 := (x_k,y_k^0)$ be the projection of $z_k$ on $\partial \Sigma_{\e_k}$ and define
$$\tilde{z}_k=(\tilde{x}_k,\tilde{y}_k) := \begin{cases}
(x_k,y_k),&\text{ in }\textbf{Case 1},\\
(x_k,y_k^0),&\text{ in }\textbf{Case 2},\\
(x_k,0),&\text{ in }\textbf{Case 3}.\\
\end{cases} $$
Define the sequence of domains 
\begin{equation*}\label{eq:domains}
\Omega_k :=\frac{B_{1}\setminus\Sigma_{\e_k}-\tilde z_k}{r_k}=\Big\{z=(x,y):|\tilde z_k+r_kz|< 1, \text{ and }|\tilde y_k+r_ky|> \e_k\Big\},
\end{equation*}
and, for every $z\in \Omega_k $, let us define the sequence of functions
\begin{align*}& v_k(z):=\frac{(\eta u_k)( \tilde z_k+r_kz)-(\eta u_k)( \tilde z_k)}{r_k^\alpha L_k },&w_k(z) := \frac{\eta(\tilde{z}_k)\big( u_k( \tilde z_k+r_kz)-u_k( \tilde z_k)\big)}{r_k^\alpha L_k }, 
\end{align*}
in \textbf{Case 1} and \textbf{Case 2}, and
\begin{align*}
 &v_k(z):=\frac{(\eta u_k)( \tilde z_k+r_kz)}{r_k^\alpha L_k },
&w_k(z) := \frac{\eta(\tilde{z}_k) u_k( \tilde z_k+r_kz)}{r_k^\alpha L_k }, 
\end{align*}
in \textbf{Case 3}.

\medskip
\noindent
\emph{Step 2. Blow-up domains.}
Let us define
\begin{equation}\label{eq:omega:infty}
\Omega_\infty := \big\{z = (x,y) \in \R^d : \text{exists }\hat k \text{ such that } z \in \Omega_k \text{ for every }k\geq \hat k\big\}.
\end{equation}
In this section we show who is the limit domain $\Omega_\infty:=\lim_{k\to \infty}\Omega_k$, along a suitable subsequence.
First, in every case, for every $z \in \R^d$, one has that 
$$|\tilde z_k+r_k z|<| z_k|+r_k|z| \le 3/4 +o(1) < 1.$$
Hence, to prove that $z\in\Omega_\infty$, we only need to show that $\tilde z_k+r_kz\not\in \Sigma_{\e_k}$, that is,
\begin{equation}\label{eq:dom:3}
|\tilde{y}_k+r_k y|> \e_k.
\end{equation}
Let us start with \textbf{Case 1}, recalling that $\tilde{z}_k=z_k$. Fix $z \in \R^d$ and by contradiction let us suppose that \eqref{eq:dom:3} does not hold. Then, since $|\cdot|$ is a Lipschitz function, one has that
\[
  \frac{|y_k| - \e_k }{r_k}  \le \frac{  |y_k|- |y_k + r_k y|}{r_k} \le c|y|,
\]
and taking the limit as $k\to  \infty$, it follows
\[
|y|\ge \infty,
\]
which is a contradiction. Hence, $\Omega_\infty = \R^d$.

\smallskip

Next, let us consider the \textbf{Case 2} and recall that $\tilde{z}_k =(x_k, y_k^0 )$, $|y_k^0| = \e_k$, $r_k/\e_k \to 0$. Defining
\[
\bar{e}:=\lim_{k\to \infty}\frac{{y}_k^0}{| {y}_k^0|},
\]
we claim that $\Omega_\infty = \Pi := \{(x,y): y\cdot \bar{e}>0\}$, which is an half-space. We observe that, for every $y \in \R^n$,
\begin{equation}\label{eq:dom:exp}
\frac{  | y_k^0 + r_k y|-| y_k^0|}{r_k} -\frac{ y_k ^0}{| y_k^0|}\cdot y   \le c \frac{r_k}{\e_k} \to 0,
\end{equation} as $k\to \infty$. Indeed, by using Lagrange's Theorem to the function $|\cdot|$, there exists $y^*$ (which could depend on $k$) such that $|y^*|\le |y|$ and denoting $y_k^*:=y_k^0+r_ky^*$, we have
\begin{align*}
&\Big|\frac{| y_k^0 + r_k y|-| y_k^0|}{r_k} - \frac{y_k^0 \cdot y   }{|y_k^0|}\Big|=\Big|\frac{y_k^*  \cdot y   }{|y_k^*|} - \frac{y_k^0  \cdot y   }{|y_k^0|}\Big| \le \Big|\frac{y_k^*  \cdot y   }{|y_k^*|} - \frac{y_k^* \cdot y   }{|y_k^0|}\Big| + \Big|\frac{y_k^* \cdot y   }{|y_k^0|} - \frac{y_k^0 \cdot y  }{|y_k^0|}\Big|\\
&\le c\Big|\frac{|y_k^0|-|y_k^*|}{|y_k^0|}\Big| + c\Big|\frac{|y_k^*-y_k^0|}{|y_k^0|}\Big|\le c\frac{r_k}{|y_k^0|}= c \frac{r_k}{\e_k} \to 0,\quad \text{as }k\to \infty,
\end{align*} 
so  \eqref{eq:dom:exp} holds true. Let us fix $z=(x,y)$ such that $\bar{e}\cdot y = \delta>0$ and suppose by contradiction that \eqref{eq:dom:3} doesn't hold, that is
\[
\frac{  | y_k^0 + r_k y|-| y_k^0|}{r_k} -\frac{ y_k ^0}{| y_k^0|}\cdot y + \frac{ y_k ^0}{| y_k^0|}\cdot y \le 0.
\]
So, by taking the limit as $k\to \infty$ and using \eqref{eq:dom:exp}, we obtain
\[
\bar{e}\cdot y \le 0,
\]
which is a contradiction. In analogous way, we have that every $z=(x,y)$ such $\displaystyle{|\bar{e}\cdot y |=-\delta <0}$ satisfies $z \not \in \Omega_k$. Hence, since $\delta>0$ is arbitrary, the claim follows, that is, $\Omega_\infty = \Pi $.

\smallskip 
Finally, let us consider the \textbf{Case 3}, recall that $\e_k< |y_k|\le c r_k$ and $\tilde{z}_k = (x_k, 0)$.  First, up to consider a subsequence, the following limit is well defined
\begin{equation}\label{eq:dom:2} 
\bar{\e}:=\lim_{k\to \infty}\frac{\e_k}{r_k}\in [0,c].
\end{equation}
Let us fix $z=(x,y)$ such that $|{y}|=\bar{\e}+\delta$ for some $\delta>0$
and suppose by contradiction that \eqref{eq:dom:3} doesn't hold. Then, 
\[
\bar{\e}+\delta = |{y}| \ge \frac{\e_k}{r_k} \to \bar{\e},
\]
which is a contradiction, so $z \in \Omega_k$. Instead, fixed $z=(x,y)$ such that $|{y}|=\bar{\e}-\delta$ for some $\delta>0$, one has  $z \not \in \Omega_k$.
Since $\delta>0$ is arbitrary, we have  that
\[
\Omega_\infty=\R^d\setminus \Sigma_{\bar \e} = \{(x,y):|y|>\bar{\e}\}.
\]
Resuming, we have shown that the limit domain is
\begin{equation}\label{eq:limit:domain}
\Omega_\infty=\begin{cases}
\R^d ,&\text{in }\textbf{Case 1},\\
\Pi,  &\text{in }\textbf{Case 2},\\
\R^d\setminus \Sigma_{\bar \e},&\text{in }\textbf{Case 3},\\
\end{cases}
\end{equation}
where $\Pi := \{(x,y)\in\R^d :\bar{e}\cdot y \ge 0\}$ is an half-space.

\medskip
\noindent
\emph{Step 3. H\"older estimates and  convergence of the blow-up sequences.}
Let us fix a compact set $K\subset \Omega_\infty$. For every $z,z'\in K$ such that $z\not=z'$, we have
\[
|v_k(z)-v_k(z')|=\frac{|(\eta u_k)(\tilde{z}_k+r_kz)-(\eta u_k)(\tilde{z}_k+r_kz')|}{r_k^\alpha L_k }\le |z-z'|^\alpha,
\]
that is,
\begin{equation}\label{eq:alpha:estimate}
[v_k]_{C^{0,\alpha}(K)}\le 1.
\end{equation}
In \textbf{Case 1} and \textbf{Case 2}, by using $v_k(0)=0$, we get the uniform bound $\|v_k\|_{C^{0,\alpha}(K)} \le c$, for every compact subset $K\subset \Omega_\infty$.
Instead, in \textbf{Case 3}, since $\e_k/r_k\le c$, one has that
\[
|v_k(x,y)|=|v_k(x,y)-v_k(0,{y_k^0}/{r_k})|\le [v_k]_{C^{0,\alpha}(K)}(|x|+|y-y_k^0/r_k|)^\alpha\le c,
\]
for some $c>0$ which depends only on $K$, where we have used the boundary condition $\eta u_k = 0$ on $\partial \Sigma_{\e_k}$. Hence,
we have that $\|v_k\|_{L^\infty(K)}\le c$, which implies that $\|v_k\|_{C^{0,\alpha}(K)}\le c$ also in this case.

\smallskip

By applying the Arzelà-Ascoli Theorem we conclude that $v_k\to\bar{v}$ uniformly in $K$. By a standard diagonal argument, we can take the the limit as $k\to \infty$ in \eqref{eq:alpha:estimate} to obtain 
\begin{equation*}\label{eq:alpha:v}
[\bar v]_{C^{0,\alpha}(\Omega_\infty)}\le 1,
\end{equation*}
which implies that $\bar v$ satisfies the growth condition
\begin{equation}\label{eq:growth:0:eps}
|\bar{v}|\le c(1+|z|^\alpha),\quad\text{a.e. in }\Omega_\infty.
\end{equation}

\smallskip

Moreover, since $u_{\e_k}= 0 $  on $\partial \Sigma_{\e_k}$, and then $v_k = 0 $ on $({\partial \Sigma_{\e_k}-\tilde{z}_k})/{r_k}$, by employing the local uniform convergence we have that $\bar{v}= 0$ on $\partial \Pi$ in \textbf{Case 2} and $\bar{v}= 0$ on $\partial \Sigma_{\bar{\e}}$ in \textbf{Case 3}.

\smallskip

Furthermore, the sequences $v_k$ and $w_k$ converge to the same limit function. Let us fix a compact set $K \subset \Omega_\infty$. For every $z \in K$, by using \eqref{eq:c0:l:infty}, we have 
\[
|v_k(z)-w_k(z)| \le \frac{(\eta (\tilde{z}_k+r_k z) -\eta (\tilde{z}_k) ) \|u_k\|_{L^\infty(B_{3/4 \setminus\Sigma_{\e_k}})} }{r_k^\alpha L_k} \le \frac{c r_k^{1-\alpha}}{L_k} \to 0, \quad \text{as }k \to \infty.
\]
Hence, the sequences $v_k$ and $w_k$ have the same asymptotic behaviour as $k\to \infty$ on every $K\subset \Omega_\infty$, which implies that $w_k \to \bar{v}$ uniformly on $K$.

\medskip
\noindent
\emph{Step 4. The limit function $\bar{v}$ is not constant.} 
Let us consider the sequences of points

\[
\xi_k^1:=\frac{{z}_k-\tilde z_k}{r_k},\quad \xi_k^2:=\frac{\hat{z}_k-\tilde z_k}{r_k}.
\]
By \eqref{eq:c0:alpha:points}, we have
\[
|v_k(\xi_k^1)-v_k(\xi_k^2)|= \frac{|(\eta u_k)(\hat{z}_k)-(\eta u_k)({z_k})|}{r_k^\alpha L_k}\ge \frac{1}{2}.
\]
In \textbf{Case 1}, we have that $$\xi_k^1=0,\quad  \xi_k^2 = \frac{\hat z_k- z_k}{r_k} \in \partial B_1,$$ 
then, $\xi_k^1\to 0$, $\xi_k^2\to \xi^2\not= 0$.

In \textbf{Case 2}, 
$$\xi_k^1=\frac{(0,y_k-y_k^0)}{r_k},\quad \xi_k^2 = \frac{(\hat x_k- x_k,\hat{y}_k-y_k^0)}{r_k}.$$ 
Since $(|y_k|-\e_k)/r_k\le c$ uniformly in $k$,
$|\xi_k^1-\xi_k^2|=1$, we have that $\xi_k^1\to\xi^1$, $\xi_k^2\to \xi^2$ and $\xi^1\not=\xi^2$.

In \textbf{Case 3}, 
$$\xi_k^1=\frac{(0,y_k)}{r_k},\quad \xi_k^2 = \frac{(\hat x_k- x_k,\hat{y}_k)}{r_k}.$$ Since $|y_k|/r_k\le c$ uniformly in $k$,
$|\xi_k^1-\xi_k^2|=1$, we have that $\xi_k^1\to\xi^1$, $\xi_k^2\to \xi^2$ and $\xi^1\not=\xi^2$.

Hence, by the local uniform convergence $v_k\to\bar{v}$ we get that $|\bar{v}({\xi^1})-\bar{v}({\xi^2})|\ge1/2$, so $\bar v$ is not constant.

\medskip
\noindent
\emph{Step 5. $\bar{v}$ is an entire solution to a homogeneous equation with constant coefficients.} 
Let us define $A_k(x,y):=A(\tilde{z}_k+r_kz)$
and $(\bar x,\bar y) = \lim_{k\to \infty}(\tilde{x}_k,\tilde{y}_k)$. By using the continuity of $A$, we can define $\bar{A}:=\lim_{k\to \infty}A_k(z)=A(\bar{x},\bar{y})$,
which is a constant coefficients symmetric matrix satisfying \eqref{eq:unif:ell}.
Next, set
$$\rho_k(y):=\begin{cases}
\displaystyle{\frac{|\tilde y_k+r_k y|}{|\tilde y_k|}}, &\text{ in }\textbf{Case 1 }\text {and } \textbf{Case 2},\\
|y|, &\text{ in }\textbf{Case 3},\\
\end{cases}$$
observing that in \textbf{Case 1} and \textbf{Case 2}
\begin{equation}\label{eq:xyx}
 |\rho_k(y)|=1+o(1),
 \end{equation}
as $k\to \infty$ on every compact subset $K\subset\R^d$.

\smallskip

Fix $\phi\in C_c^\infty(\Omega_\infty)$.
since $u_k$ is a weak solution to \eqref{eq:sol:c0:k}, a straightforward computation shows us that
\begin{align}\label{eq:sol:vk}
\begin{aligned}
&\int_{\supp(\phi)}\rho_k^a(y) A_k(z)\D w_k(z)\cdot \D {\phi}(z) dz
=
\frac{r_k^{2-\alpha} \eta(\tilde{z}_k)}{L_k} \int_{\supp(\phi)}\rho_k^a(y) f(\tilde{z}_k+r_kz){\phi}(z) dz \\
\qquad& - 
\frac{r_k^{1-\alpha} \eta(\tilde{z}_k)}{L_k}
 \int_{\supp(\phi)}\rho_k^a(y)F(\tilde{z}_k+r_kz)\cdot \D {\phi}(z) dz .
 \end{aligned}
\end{align}
Now, we aim to prove that the right-hand side of \eqref{eq:sol:vk} vanishes as $k\to \infty$. First, we focus on the term involving the function $f$.
Let us consider the \textbf{Case 1} and \textbf{Case 2} together.  By using the H\"older inequality, \eqref{eq:xyx} and $a<0$, we get
\begin{align*}
&\Big|\int_{\supp(\phi)}\frac{|\tilde y_k+r_ky|^a}{|\tilde y_k|^a} f(\tilde z_k+r_kz){\phi}(z) dz \Big|\\
&\le 
c\|\phi\|_{L^\infty}\Big(r_k^{-d}|\tilde{y}_k|^{-a}\int_{B_1}|y|^a|f|^p dz \Big)^{1/p}\Big(\int_{\supp(\phi)}
\frac{|\tilde y_k+r_ky|^a}{|\tilde y_k|^a}
 dz \Big)^{1/p'}
\le c r_k^{{d}/{p}}.
\end{align*}
Hence,
\begin{equation*}
\frac{r_k^{2-\alpha} \eta(\tilde{z}_k)}{L_k}\Big| \int_{\supp(\phi)}\rho_k^a(y) f(\tilde z_k+r_kz){\phi}(z)dz \Big|\le c r_k^{2-\alpha-{d}/{p}}L_k^{-1}\to0,
\end{equation*}
as $k\to \infty$, by the hypothesis \eqref{eq:alpha:0}, which implies $\alpha\le 2-{d}/{p}$.

\smallskip

Let us consider now the \textbf{Case 3}. By using the H\"older inequality and $a<0$, we get
\begin{align*}
&\Big|\int_{\supp(\phi)}|y|^a f(\tilde z_k+r_kz){\phi}(z)dz\Big|\\
&\le c\|\phi\|_{L^\infty}\Big(r_k^{-d-a}\int_{B_1}|y|^a|f|^p dz \Big)^{1/p}\Big(\int_{\supp(\phi)}
|y|^a dz 
\Big)^{1/p'}
\le c r_k^{-{d}/{p}},
\end{align*}
and then
\begin{equation*}
\frac{r_k^{2-\alpha} \eta(\tilde{z}_k)}{L_k}\Big| \int_{\supp(\phi)}|y|^a f(\tilde z_k+r_kz){\phi}(z) dz\Big|\le c r_k^{2-\alpha-{d}/{p}}L_k^{-1}\to0,
\end{equation*}
as before.

\smallskip

The second member of the the right hand side of \eqref{eq:sol:vk} vanishes as $k\to \infty$ by using similar computations. In fact, in every cases, we have
\begin{equation*}
\frac{r_k^{1-\alpha} \eta(\tilde{z}_k)}{L_k}\Big| \int_{\supp(\phi)}\rho_k^a(y) F(\tilde z_k+r_kz)\cdot\D {\phi}(z)dz \Big|\le c r_k^{1-\alpha-{d}/{q}}L_k^{-1}\to0,
\end{equation*}
by the assumption \eqref{eq:alpha:0}, which implies that $\alpha \le 1- d/q$.

\smallskip

Finally, we prove that the left hand side of \eqref{eq:sol:vk} converges in the following sense

\begin{equation}\label{eq:l.h.s.}
\int_{\supp(\phi)}\rho_k^a(y) A_k(z)\D w_k (z)\cdot \D {\phi}(z)dz\to \int_{\supp(\phi)}\bar{\rho}^a(y) \bar{A}\D \bar{v}(z)\cdot \D \phi(z)dz, 
\end{equation}
where $$\bar\rho(y):=\begin{cases}
1,&\text{ in }\textbf{Case 1 }\text{and } \textbf{Case 2},\\
|y|,&\text{ in }\textbf{Case 3}.
\end{cases}$$
Let us fix $R>0$ such that $\supp(\phi)\subset B_{2R}\cap \Omega_k $ for every $k$ large enough.
Since $w_k$ is uniformly bounded in $L^\infty(B_{2R}\cap \Omega_\infty)$, one has that $w_k$ is uniformly bounded in $L^{2}(B_{2R}\cap \Omega_\infty,\rho_k^a(y)dz)$ and, by applying the Caccioppoli-type inequality \eqref{eq:caccioppoli}, we get that  $w_k$ is uniformly bounded in $H^{1}(B_{2R}\cap \Omega_\infty,\rho_k^a(y)dz)$. Using $\rho_k^a\to\bar \rho$ and $A_k\to\bar{A}$ a.e. in $\Omega_\infty$ and arguing as in the proof of Lemma \ref{lem:approximation}, we can conclude that \eqref{eq:l.h.s.} holds true and $\bar{v}\in H^1(B_{R}\cap \Omega_\infty, \bar \rho^a )$.

\smallskip

Hence, recalling the definition of the limit domain $\Omega_\infty$, see \eqref{eq:limit:domain}, and using \emph{Step 3} for the boundary condition, we conclude that
\begin{itemize}[left=0pt]
\item in \textbf{Case 1}, $\bar{v}$ is an entire solution to
$$-\div(\bar{A}\D\bar{v})=0,\quad\text{in }\R^d,$$

\item in \textbf{Case 2}, $\bar{v}$ is an entire solution to

$$\begin{cases}
-\div(\bar{A}\D\bar{v})=0, & \text{in }\Pi ,\\
\bar{v}=0, & \text{on } \partial\Pi,
\end{cases}$$
\item in \textbf{Case 3}, $\bar{v}$ is an entire solution to

$$\begin{cases}
-\div(|y|^a\bar{A}\D\bar{v})=0, & \text{in }\R^d\setminus\Sigma_{\bar{\e}},\\
\bar{v}=0 ,& \text{on } \partial\Sigma_{\bar{\e}}.
\end{cases}$$

\end{itemize}

\medskip
\noindent
\emph{Step 6. Liouville Theorems and conclusion.}
By \eqref{eq:growth:0:eps} we have that $\bar v $ satisfies the growth condition
$$|\bar{v}(z)|\le c(1+|z|^\alpha),$$ for every $z\in \Omega_\infty$, where $\alpha<\min\{1,2-a-n\}$ by hypothesis \eqref{eq:alpha:0}. 
In \textbf{Case 1} and \textbf{Case 2}, by invoking the classical Liouville Theorem, we can conclude that $\bar v$ must be constant and this is a contradiction since $\bar v$ is not constant by \emph{Step 4}.
In \textbf{Case 3}, we have that $\bar{v}$ satisfies the hypothesis of the Liouville Theorem \ref{teo:liouville} and so $\bar{v}$ must be identically zero, which is a contradiction.
Then, $L_k \le c$ uniformly in $k$, which implies that $[u_{\e}]_{C^{0,\alpha}(B_{1/2}\setminus\Sigma_\e)}\le c$. The proof is complete.
\end{proof}

\begin{proof}[Proof of the Theorems \ref{teo:0}]

Let $u$ be a weak solution to \eqref{eq:1} and consider the trivial extension of $\psi$ in $B_1$, that is, $\psi(x,y)=\psi(x)$, for every $(x,y)\in B_1$. Let us define 
$$v:=u-\psi. $$
Since $\psi$ is a Lipschitz function, we have that $ v \in \tilde H^{1,a}(B_1)$ and $v$ is a weak solution to 
\[
\begin{cases}
-\div(|y|^a A \nabla v)= |y|^a f+\div(|y|^a (F-A\D \psi)), & \text{in }B_1\setminus \Sigma_{0},\\
v=0, & \text{on } \Sigma_0\cap B_1.
\end{cases}
\]
By applying Lemma \ref{lem:approximation} we find a sequence $\{v_{\e_k}\}$ as $\e_k\to0$, such that every $v_{\e_k}$ is solution to 
\[
\begin{cases}
-\div(|y|^a A \nabla v_{\e_k})= |y|^a f+\div(|y|^a(F-A\D\psi)), & \text{in }B_{3/4}\setminus \Sigma_{\e_k},\\
v_{\e_k}=0, & \text{on } \partial \Sigma_{\e_k}\cap B_{3/4},
\end{cases}
\]
and $v_{\e_k}\to v$ in $H^{1,a}(B_{3/4})$ as $\e_k\to0$. By applying the Theorem \ref{teo:0alpha:e} to the sequences $\{v_{\e_k}\}$, combined with \eqref{eq:stima:rhs}, we get that
\begin{equation}\label{eq:teo:0:a}
\|v_{\e_k}\|_{C^{0,\alpha}(B_{1/2}\setminus\Sigma_{\e_k})}\le c \big(
\|u\|_{L^{2,a}(B_{1})}
+\|f\|_{L^{p,a}(B_{1})}
+\|F\|_{L^{q,a}(B_{1})^d}
+\| \psi\|_{C^{0,1}(\Sigma_0\cap B_{1})}
\big),
\end{equation}
for some $c>0$ depending only on $d$, $n$, $a$, $\lambda$, $\Lambda$, $p$, $q$, $\alpha$ and $L$.

By applying Arzelà-Ascoli Theorem, we get that ${v}_{\e_k}\to w$ in $C_{\loc}^{0,\gamma}(B_{1/2}\setminus\Sigma_0)$, for every $\gamma\in(0,\alpha)$, and by the a.e. convergence $v_{\e_k}\to v$ it follows that $v=w$. Furthermore, by taking $z,z'\in B_{1/2}\setminus\Sigma_0$ such that $z\not=z'$, we have that
\begin{align*}
&\frac{|v(z)-v(z')|}{|z-z'|^\alpha}=\lim_{\e_k\to0}\frac{|v_{\e_k}(z)-v_{\e_k}(z')|}{|z-z'|^\alpha}\\
&\le c \big(
\|u\|_{L^{2,a}(B_{1})}
+\|f\|_{L^{p,a}(B_{1})}
+\|F\|_{L^{q,a}(B_{1})^d}
+\| \psi\|_{C^{0,1}(\Sigma_0\cap B_{1})}
\big),
\end{align*}
which implies that 
\begin{align*}
&[v]_{C^{0,\alpha}(B_{1/2}\setminus\Sigma_0)}= \sup_{\substack{z,z'\in B_{1/2}\setminus\Sigma_0\\ z \not= z'}}\frac{|v(z)-v(z')|}{|z-z'|^\alpha}\\
&\le c \big(
\|u\|_{L^{2,a}(B_{1})}
+\|f\|_{L^{p,a}(B_{1})}
+\|F\|_{L^{q,a}(B_{1})^d}
+\| \psi\|_{C^{0,1}(\Sigma_0\cap B_{1})}
\big).
\end{align*}
By continuity, we can extend $v$ to the entire $B_{1/2}$ in such a way that $[v]_{C^{0,\alpha}(B_{1/2})} = [v]_{C^{0,\alpha}(B_{1/2} \setminus \Sigma_0)}$.
Finally, combining the previous estimates with the $L^\infty_{\loc}$ bounds of solutions (see Lemma \ref{lem:moser}), we obtain
\begin{align*}
&\|u\|_{C^{0,\alpha}(B_{1/2})}\le \|v\|_{C^{0,\alpha}(B_{1/2})}+\|\psi\|_{C^{0,\alpha}(\Sigma_0\cap B_{1/2})}\\
&\le c \big(
\|u\|_{L^{2,a}(B_{1})}
+\|f\|_{L^{p,a}(B_{1})}
+\|F\|_{L^{q,a}(B_{1})^d}
+\| \psi\|_{C^{0,1}(\Sigma_0\cap B_{1})}
\big).
\end{align*}
that is, $u \in C^{0,\alpha}(B_{1/2})$ and \eqref{eq:c0} holds true. 
\end{proof}

\section{Schauder estimates for weak solutions}\label{section:schauder}

This section is devoted to the proof of the Theorem \ref{teo:1}, which establishes the $C^{1,\alpha}_\loc$ regularity for weak solutions. To achieve this result, we proceed as follows: first, as in Theorem \ref{teo:0alpha:e}, we prove $\e$-uniform estimates for solutions in perforated domains, with an additional assumption on the field $F$. As we will see in Remark \ref{R:F:radial}, this condition cannot be removed. Next, we show \emph{a priori} estimates for solutions, which also satisfy an additional boundary condition on $\Sigma_0$. Afterwards, we establish the main result through a double approximation process: the first involves perforated domains and using Lemma \ref{lem:approximation}, while the second one is a standard approximation via convolution with a family of mollifiers. 

\begin{teo}\label{teo:1alpha:e}
Let $2\le n\le d$, $a+n\in(0,1)$, $p>{d}$ and $\alpha$ satisfying \eqref{eq:alpha:1}.
Let $A$ be a $\alpha$-H\"older continuous matrix satisfying \eqref{eq:unif:ell} and $\|A\|_{C^{0,\alpha}{(B_1)}}\le L$, $f\in L^{p,a}(B_1)$, $F\in{C^{0,\alpha}{(B_1)}}$ be a field such that $F(x,0)\cdot e_{y_i}=0$ for every $(x,0)\in B_1$ and for every $i=1,\dots,n$. For $0<\e\ll1$, let $\{u_\e\}$ be a family of solutions to \eqref{eq:sol:c0}.

Then, there exists a constant $c>0$, depending only on $d$, $n$, $a$, $\lambda$, $\Lambda$, $p$, $\alpha$ and $L$ such that
\begin{equation}\label{eq:stima:c1:e}
\|u_\e\|_{C^{1,\alpha}(B_{1/2}\setminus\Sigma_\e)}\le c\big(
\|u_\e\|_{L^{2,a}(B_{1}\setminus\Sigma_\e)}
+\|f\|_{L^{p,a}(B_{1})}
+\|F\|_{C^{0,\alpha}(B_{1})}
\big).
\end{equation}
In addition, $u_\e$ satisfies
\begin{equation}\label{eq:formal:conormal:eps}
|\D u_\e (z)|\le c\big(
\|u_\e\|_{L^{2,a}(B_{1}\setminus\Sigma_\e)}
+\|f\|_{L^{p,a}(B_{1})}
+\|F\|_{C^{0,\alpha}(B_{1})}
\big)\e^\alpha, \quad \text{for every }z\in\partial \Sigma_\e \cap B_{1/2}.
\end{equation}

\end{teo}

\begin{proof}
Under the assumptions of the theorem, classical Schauder theory ensures that solutions to \eqref{eq:sol:c0} are $C^{1,\alpha}(B_{1/2}\setminus\Sigma_\e)$ and that \eqref{eq:stima:c1:e} holds with a constant $c>0$ that may also depend on $\e$. Our goal is to show that it is possible to provide a constant $c>0$ that not depends on $\e$.

Without loss of generality, we can assume that 
\[
\|u_\e\|_{L^{2,a}(B_{1}\setminus\Sigma_\e)}
+\|f\|_{L^{p,a}(B_{1})}
+\|F\|_{L^{q,a}(B_{1})^d}\le c,
\]
for some $c>0$, which not depends on $\e$. Moreover, for every $\beta \in (0,1)$, the assumptions of Theorem \ref{teo:0alpha:e} are satisfied, so
\begin{equation}\label{eq:c1:l:infty}
\|u_\e\|_{C^{0,\beta}(B_{3/4}\setminus\Sigma_\e)}\le c,
\end{equation}
for some $c>0$, which not depends on $\e$.

\medskip
\noindent
\emph{Step 1. Contradiction argument and blow-up sequences.} 
By contradiction let us suppose that there exists $p>d$, $\alpha$ satisfying \eqref{eq:alpha:1}, $\{u_k\}_k:=\{u_{\e_k}\}_k$, as $\e_k\to0$ such that 
$u_k$ is solution to \eqref{eq:sol:c0:k} and \[
\|\D u_k\|_{C^{0,\alpha}(B_{1/2}\setminus\Sigma_{\e_k})}\to \infty.
\]
Let us fix a smooth cut-off function $\eta \in C_c^\infty(B_1)$ such that
\[
\supp(\phi)\subset B_{3/4},
\quad 
\eta=1 \text{ in }B_{1/2},
\quad
0\le\eta\le1.
\]
One has that 
\begin{equation*}
\|\eta u_k\|_{C^{1,\alpha}(B_{1}\setminus\Sigma_{\e_k})} \to \infty.
\end{equation*}
Let us define 
\begin{equation*}
L_k:=[\D(\eta u_k)]_{C^{0,\alpha}(B_{1}\setminus\Sigma_{\e_k})},
\end{equation*}
and notice that it cannot be possible that $[\D (\eta u_k)]_{C^{0,\alpha}(B_1)}\le c$ and $\| \D (\eta u_k)\|_{L^\infty(B_1)}\to \infty$, since $\D (\eta u_k) = 0$ outside of $ B_{3/4}\setminus\Sigma_{\e_k}$. Hence, $L_k\to \infty$.

By definition of H\"older seminorm, take two sequences of points $z_k=(x_k,y_k),\hat{z}_k=(\hat x_k,\hat y_k)\in B_{1}\setminus\Sigma_{\e_k}$ such that
\begin{equation}\label{eq:c1:alpha:points}
\frac{|\D(\eta u_k)(z_k)-\D(\eta u_k)(\hat z_k)|}{|z_k-\hat z_k|^\alpha}\ge\frac{L_k}{2},
\end{equation}
define $r_k:=|z_k-\hat z_k|$ and observe that at least one of $z_k$ or $\hat z_k$ belongs to $B_{3/4}\setminus\Sigma_{\e_k}$.
We distinguish three cases.
 
\smallskip

\begin{itemize}[left=0pt]
    \item \textbf{Case 1:} $ 
    \displaystyle{
    \frac{|y_k|}{r_k}\to \infty, \quad \frac{|y_k|-\e_k}{r_k}\to \infty,}$
    
        \item \textbf{Case 2:} $ 
    \displaystyle{
    \frac{|y_k|}{r_k}\to \infty, \quad \frac{|y_k|-\e_k}{r_k}\le c},
     $
    
\item \textbf{Case 3:}
    $ 
    \displaystyle{
    \frac{|y_k|}{r_k}\le c, }$
    
\end{itemize}
for some constant $c>0$ which not depends on $k$. We notice that $r_k \to 0$ in \textbf{Case 1} and \textbf{Case 2} and we show later that $r_k\to 0$ also in \textbf{Case 3}.

\smallskip

Let $z_k^0 := (x_k,y_k^0)$ be the projection of $z_k$ on $\partial \Sigma_{\e_k}$, define
$$\tilde{z}_k=(\tilde{x}_k,\tilde{y}_k) := \begin{cases}
(x_k,y_k),&\text{ in }\textbf{Case 1},\\
(x_k,y_k^0),&\text{ in }\textbf{Case 2},\\
(x_k,0),&\text{ in }\textbf{Case 3},\\
\end{cases} $$
and the sequence of domains 
\begin{equation*}
\Omega_k :=\frac{B_{1}\setminus\Sigma_{\e_k}-\tilde z_k}{r_k}.
\end{equation*}
For every $z\in \Omega_k $, let us define the sequence of functions

\begin{align*}
&v_k(z):= \frac{
(\eta u_k)( \tilde{z}_k+r_kz) -(\eta u_k )(\tilde{z}_k) - \D(\eta u_k)(\tilde{z}_k)\cdot r_k z
}{L_k r_k^{1+\alpha}},\\
&w_k(z):=\frac{\eta(\tilde{z}_k) u_k( \tilde{z}_k+r_kz) -  (\eta u_k) (\tilde{z}_k) -(\eta \D  u_k)(\tilde{z}_k)\cdot r_k z}{L_k r_k^{1+\alpha}},
\end{align*}
in \textbf{Case 1} and \textbf{Case 2}, and
\begin{align*}
v_k(z):= \frac{
(\eta u_k)( \tilde{z}_k+r_kz)}{L_k r_k^{1+\alpha}}, \quad w_k(z):=\frac{\eta(\tilde{z}_k)  u_k( \tilde{z}_k+r_kz) }{L_k r_k^{1+\alpha}},
\end{align*}
in \textbf{Case 3}. Furthermore, let us define the limit domain $\Omega_\infty $ as in \eqref{eq:omega:infty}.

\medskip
\noindent
\emph{Step 2. Gradient H\"older estimates and convergence of the blow-up sequences.} 
Let us fix a compact set $K\subset \Omega_\infty$. For every $z,z'\in K$ such that $z\not=z'$ we have
\[
|\D v_k(z)-\D v_k(z')|=\frac{|\D (\eta u_k)(\tilde{z}_k+r_kz)-\D (\eta u_k)(\tilde{z}_k+r_kz')|}{r_k^\alpha L_k }\le |z-z'|^\alpha,
\]
that is
\begin{equation}\label{eq:alpha:estimate:1}
[\D v_k]_{C^{0,\alpha}(K)}\le 1.
\end{equation}
In \textbf{Case 1} and \textbf{Case 2}, since  $v_k(0)=0$ and $|\D v_k(0)|=0$, we get the uniform bound of the norm 
\[\|v_k\|_{C^{1,\alpha}(K)}\le c.\]
In \textbf{Case 3} we use a different argument to show a uniform bound of $\|v_k\|_{C^{1,\alpha}(K)}$. First, for every point $z'_k=(x_k',y_k')\in \partial\Sigma_{\e_k}\cap B_{3/4}$, one has that 
one has that $(\eta u_k)(z'_k)=0$ and, recalling that the normal vector to $ \partial \Sigma_{\e_k}$ at the point ${z}'_k$ is the vector $(0,y'_k)/|y'_k| \in \mathbb{S}^{d-1}$, it follows that
\[
\D (\eta u_k) (z'_k)\cdot \vec e =0,\quad \text{for every vector } \vec e\perp (0,y'_k).
\]
Let us fix $i,h\in \{1,\dots,n\}$ such that $i\not=h$. Then, 
\begin{align}\label{eq:D:vk0:y}
\begin{split}
&\big|\D (\eta u_k) (x_k',y_k')\cdot e_{y_i}\big|=\big|\big(\D (\eta u_k) (x_k',y_k')-\D (\eta u_k) (x_k',\e_k e_{y_h})\big)\cdot e_{y_i}\big|\\&
\le [\D(\eta u_k)]_{C^{0,\alpha}(B_1\setminus\Sigma_{\e_k})} |y_k'-\e_k e_{y_h}|^\alpha\le 2^\alpha L_k\e_k^\alpha,
\end{split}
\end{align}
and 
\begin{equation}\label{eq:D:vk0:x}
|\D (\eta u_k) (x_k',y_k')\cdot e_{x_j}|=0,\quad \text{for every }j=1,\dots,d-n.
\end{equation}
Hence, by using \eqref{eq:D:vk0:y} and \eqref{eq:D:vk0:x}, it follows that
\begin{equation}\label{eq:stima:grad:bordo}
|\D (\eta u_k ) (x_k',y_k')|\le c L_k \e_k^\alpha.
\end{equation}
By using interpolation inequality in H\"older spaces \cite{GT01}*{Lemma 6.35}, one has 
$$\|v_k\|_{C^{1,\alpha}(K)}\le c( \|v_k\|_{L^\infty(K)}+[v_k]_{C^{1,\alpha}(K)}),$$
and, by using the first order expansion of $v_k$ at the point $(0,y_k^0/r_k)$ we obtain
\begin{align}\label{eq:vk:first:order}
\begin{split}
|v_k(x,y)|&\le|v_k(x,y)-v_k(0,y_k^0/r_k)-\D v_k(0,y_k^0/r_k)\cdot z|+|\D v_k(0,y_k^0/r_k)|\\
&\le  c(|x|+|y-y_k^0/r_k|)^{1+\alpha}+|\D v_k(0,y_k^0/r_k)|\le c+|\D v_k(0,y_k^0/r_k)|,
\end{split}
\end{align}
for some $c>0$ which depends only on $K$, noticing that $|y_k^0/r_k|=\e_k/r_k\le c$ uniformly in $k$ in \textbf{Case 3}.
So, if 
\begin{equation}\label{eq:D:vk(0)}
|\D v_k(0,y_k^0/r_k)|\le c,
\end{equation}
uniformly in $k$, we can take the $\sup_K$ in \eqref{eq:vk:first:order} to obtain a uniform of $\|v_k\|_{L^\infty(K)}$ which implies $\|v_k\|_{C^{1,\alpha}(K)}\le c$.
Then, by using \eqref{eq:stima:grad:bordo} and recalling that $(x_k,y_k^0) \in \partial\Sigma_{\e_k}\cap B_{3/4}$, we get that
\[
|\D v_k(0,y_k^0/r_k)|=\frac{|\D (\eta u_k) (x_k,y_k^0)|}{r_k^\alpha L_k}\le c\frac{\e_k^\alpha}{r_k^\alpha}\le c.
\]
So, \eqref{eq:D:vk(0)} holds true and $\|v_k\|_{C^{1,\alpha}(K)}\le c$ uniformly in $k$.

\smallskip

Then, we may apply the Arzelà-Ascoli Theorem to infer that $v_k\to\bar{v}$ in $ C^{1,\gamma}(K)$ for any $\gamma\in(0,\alpha)$. By a standard diagonal argument, we can take the the limit as $k\to\infty$ in \eqref{eq:alpha:estimate:1} to obtain
\begin{equation*}
[\D \bar v]_{C^{1,\alpha}(\Omega_\infty)}\le 1,
\end{equation*} 
which implies that
 \begin{equation}\label{eq:eps:growth:1}
 |\bar{v}(z)|\le c(1+|z|^{1+\alpha}),\quad \text{a.e. in }\Omega_\infty.
 \end{equation}

\smallskip

Moreover, $v_k$ and $w_k$ converge to the same limit function $\bar{w}$.
Let us fix a compact set $K \subset \Omega_\infty$. In \textbf{Case 1} and \textbf{Case 2}, for every $z \in K$, by using \eqref{eq:c1:l:infty} and exploiting the first order expansion of $\eta$, we have 
\begin{align*}
&|v_k(z)-w_k(z)|= \frac{|(\eta u_k)(\tilde{z}_k+r_k z) -\eta(\tilde{z}_k)u_k (\tilde{z}_k+r_k z) -( u_k \D \eta) (\tilde{z}_k) \cdot r_k z  |}{r_k^{1+\alpha} L_k}\\
&\le 
\frac{|u_k(\tilde{z}_k+r_k z)(\eta(\tilde{z}_k+r_k z)  -\eta(\tilde{z}_k) - \D \eta(\tilde{z}_k)\cdot r_k z)|}{r_k^{1+\alpha} L_k} + c
\frac{|\D \eta(\tilde{z}_k)| |u_k(\tilde{z}_k+r_k z)-u_k(\tilde{z}_k)|}{r_k^\alpha L_k} \\
& \le c
\frac{\|u_k\|_{L^\infty(B_{3/4}\setminus\Sigma_{\e_k})} r_k^{1-\alpha}}{L_k} + c \frac{\|u_k\|_{C^{0,\beta}(B_{3/4}\setminus\Sigma_{\e_k})} r_k^{\beta-\alpha}}{L_k} \le
 \frac{c r_k^{\beta-\alpha}}{L_k} \to 0,
\end{align*}
as $k\to\infty$, since we can choose $\beta \in (\alpha,1)$.
In \textbf{Case 3}, one has that 
\begin{align*}
&|v_k(z)-w_k(z)|= \frac{|\eta (\tilde{z}_k+r_k z) -\eta (\tilde{z}_k) | \cdot | u_k (\tilde{z}_k+r_k z)| }{r_k^{1+\alpha} L_k} \\
&\le \frac{c |u_k (\tilde{z}_k+r_k z)- u_k(z_k^0)| }{r_k^{\alpha}L_k}
\le \frac{c \|u_k\|_{C^{0,\beta}(B_{3/4}\setminus\Sigma_{\e_k})}|\tilde{z}_k+r_k z- z_k^0|^\beta }{r_k^{\alpha}L_k}
\le \frac{r_k^{\beta-\alpha}}{L_k}\to 0,
\end{align*}
where we have used the facts that $u_k(z_k^0) =0$ and $|\tilde{z}_k+r_k z- z_k^0|\le r_k|x| + r_k|y|+|y_k^0|\le c r_k + \e_k\le c r_k$ in \textbf{Case 3}.
Hence, the sequences $v_k$ and $w_k$ have the same asymptotic behaviour as $k\to\infty$ on every $K\subset \Omega_\infty$, which implies that $w_k \to \bar{v}$ uniformly on $K$.

\medskip
\noindent
\emph{Step 3.  $\D \bar{v}$ is not constant.}
Let us define the sequence of points
\[
\xi^1_k:=\frac{{z}_k-\tilde{z}_k}{r_k},\quad \xi^2_k:=\frac{\hat{z}_k-\tilde{z}_k}{r_k}.
\]
By using \eqref{eq:c1:alpha:points} we get
\begin{equation}\label{eq:grad:not:constant}
|\D v_k(\xi_k^1)-\D v_k(\xi_k^2)|=\frac{|\D (\eta u_k)({z}_k)-(\D\eta u_k)(\hat{z}_k)|}{r_k^\alpha L_k}\ge\frac{1}{2}>0.
\end{equation}
Arguing as Theorem \ref{teo:0alpha:e}, \emph{Step 4}, we  have that $\xi^1_k\to \xi_1$ and $ \xi^2_k\to \xi_2$ and $\xi^1\not=\xi^2$. Since $\D v_k\to \D\bar{v}$ uniformly on compact set by \emph{Step 2}, we can take the limit in 
\eqref{eq:grad:not:constant} to obtain $|\D\bar{v}({\xi}^1)- \D \bar{v}({\xi}^2)|>\delta_0/2$.

\medskip
\noindent
\emph{Step 4.  $r_k\to 0$ in} \textbf{Case 3}.  By contradiction let us suppose that $r_k\to \bar{r}>0$. Fixed $z \in \Omega_\infty$, we have that
\begin{align*}
&|\bar{v}(z)| = \Big| \lim_{k\to\infty} v_k(z)\Big| = \Big|\lim_{k\to\infty} \frac{
(\eta u_k)( \tilde{z}_k+r_kz) 
}{L_k r_k^{1+\alpha}} \Big| \le \frac{2\|u_k\|_{L^\infty(B_{3/4}\setminus\Sigma_{\e_k})}}{r_k^{1+\alpha }L_k} \le \frac{c}{L_k}\to 0,
\end{align*}
hence, $\bar{v} = 0$, which is a contradiction with \emph{Step 3}, where we have proved that $\D \bar{v}$ is not constant. Then, $r_k\to 0$. By arguing as in Theorem \ref{teo:0alpha:e}, \emph{Step 3}, we have that the limit domain $\Omega_\infty$ is defined by \eqref{eq:limit:domain}.

\medskip
\noindent
\emph{Step 5.  $\bar{v}$ satisfies a homogeneous Dirichlet boundary condition in} \textbf{Case 2} \emph{and} \textbf{Case 3}. First,
in \textbf{Case 3}, since $v_k=0$ on ${\partial \Sigma_{\e_k/r_k}}$ and $v_k\to\bar v$ uniformly on every compact set $K\subset \Omega_\infty$, one can conclude that $\bar{v}= 0$ on  $\partial \Sigma_{\bar{\e}}$. 

In \textbf{Case 2}, recalling that $\tilde{z}_k=z_k^0 \in \partial \Sigma_{\e_k}$, let us fix a boundary point $z \in{(\partial \Sigma_{\e_k}-{y}_k^0)}/{r_k}$ and denote by $z^\perp=(x,y^\perp)$ the projection of $z$ on the hyperplane $\Pi_k=\{ y\cdot {y}_k^0 = 0\}$. Recalling \eqref{eq:dom:exp} and observing that $|{y}_k^0 + r_k y| = | {y}_k^0 | = \e_k$, one has that
\begin{equation}\label{eq:c1:projection}
|y-y^\perp| = 
\Big| 
y\cdot \frac{ y_k^0 }{| y_k^0|} \Big| = \Big| \frac{|{y}_k^0 +r_k y| - |{y}_k^0|}{r_k} - 
y\cdot \frac{ y_k^0 }{| y_k^0|}
\Big|
\le c 
\frac{r_k}{\e_k}.
\end{equation}
Then, by using \eqref{eq:c1:projection}, \eqref{eq:stima:grad:bordo}, and noting that $\D (\eta u_k)({z}_k^0)\cdot y = \D (\eta u_k) ({z}_k^0)\cdot (y-y^\perp)$, we obtain
\begin{align*}
|v_k(z)| = \frac{|\D (\eta u_k)({z}_k^0)\cdot y|}{r_k^\alpha L_k}\le 
\frac{|\D(\eta u_k) ({z}_k^0)| }{r_k^\alpha  L_k}  |y-y^\perp| \le 
c\Big(\frac{\e_k}{r_k}\Big)^\alpha 
\frac{r_k}{\e_k}
 \to 0,
\end{align*}
as $k\to\infty$. Therefore, since $v_k\to\bar v$ uniformly on every compact set $K\subset \Omega_\infty$, we conclude that $\bar{v}= 0$ on $\partial \Pi$.

\medskip
\noindent
\emph{Step 6. $\bar{v}$ is an entire solution to a homogeneous equation with constant coefficients.} 
Let $A_k(z):=A(\tilde{z}_k+r_kz)$ and $\bar{z}:=\lim_{k\to\infty}\tilde{z}_k$. By using the $\alpha$-H\"older continuity of $A$, we can define $\bar{A}:=A(\bar{z})=\lim_{k\to\infty}A_k(z)$, which is a constant coefficients symmetric matrix satisfying \eqref{eq:unif:ell}.
Let us define
$$\rho_k(y):=\begin{cases}
\displaystyle{\frac{|y_k+r_k y|}{|y_k|}}, &\text{ in }\textbf{Case 1 }\text{and } \textbf{Case 2},\\
|y|, &\text{ in }\textbf{Case 3},\\
\end{cases}$$
and fix $\phi\in C_c^\infty(\Omega_\infty)$. Since $u_k$ is a solution to \eqref{eq:sol:c0:k}, we have that

\begin{align}\label{eq:vk:c1:limit}
\begin{aligned}
&\int_{\supp(\phi)}\rho_k^a(y) A_k (z)\D w_k(z)\cdot \D \phi (z) dz
= \frac{\eta(\tilde{z}_k) r_k^{1-\alpha}}{L_k} \int_{\supp(\phi)}\rho_k^a(y) f(\tilde{z}_k+r_k z)\phi(z)dz\\
&- \frac{\eta(\tilde{z}_k) r_k^{-\alpha}}{L_k} \int_{\supp(\phi)}\rho_k^a(y) F(\tilde{z}_k+r_k z)\cdot \D\phi(z)dz
-
\frac{r_k^{-\alpha}}{L_k} \int_{\supp(\phi)}\rho_k^a(y) A_k(z)P_k\cdot \D\phi(z)dz\\&= \text{I}+\text{II}+\text{III},
 \end{aligned}
\end{align}
where we have set
$$P_k:=\begin{cases}
 (\eta \D u_k)(\tilde z_k),&\text{ in } \textbf{Case 1 }\text{and } \textbf{Case 2},\\
0,&\text{ in } \textbf{Case 3}.
\end{cases}$$

We want to show that the right hand side vanishes as $k\to\infty$. 
The term \text{I} vanishes exactly as in the Theorem \ref{teo:0alpha:e}, \emph{Step 5}, by using the integrability assumption $f\in L^{p,a}(B_1)$, with $p>d$ and $\alpha\in (0,1-{d}/{p}]$ by \eqref{eq:alpha:1}.

\smallskip

Next, by using the divergence theorem, we have
\begin{align}\label{eq:div:F:radal}
\begin{split}
&\Big|\int_{\supp(\phi)}\rho_k^a(y) F(\tilde{z}_k+r_k z)\cdot \D\phi(z)dz\Big|\\
&\le \int_{\supp(\phi)}\rho_k^a(y) |F(\tilde{z}_k+r_k z)-F(\tilde{z}_k)|| \D\phi(z)|dz
+
\Big|\int_{\supp(\phi)}\rho_k^a(y) F(\tilde{z}_k)\cdot \D\phi(z)dz\Big|\\
&\le
c \|F\|_{C^{0,\alpha}(B_1)}r_k^\alpha + \Big|\int_{\supp(\phi)}\D \rho_k^a(y)\cdot F(\tilde{z}_k)\phi(z)dz\Big|\\
&\le cr_k^\alpha + \Big|\int_{\supp(\phi)}\D \rho_k^a(y)\cdot F(\tilde{z}_k)\phi(z)dz\Big|.
\end{split}
\end{align}
In \textbf{Case 3}, since $F(x,0)\cdot e_{y_i} = 0 $, one has that 
 $$\D|y|^a \cdot F(\tilde{z}_k)=a|y|^{a-2} F(x_k,0)\cdot y =0,$$ for every $z\in \supp(\phi)$, so the second term in \eqref{eq:div:F:radal} is zero and 
 \[
|\text{II}|\le \frac{ c }{L_k} 
\to 0,\quad \text{as }k \to\infty.
\]
Instead, in \textbf{Case 1} and \textbf{Case 2}, since $|\tilde y_k+r_k y|\ge |\tilde y_k|/2$ for $y\in\supp(\phi)$, we have that
\begin{align}\label{eq:II:term:case1}
\begin{split}
\big|\D\rho_k^a(y)\cdot F(\tilde{z}_k)\big|&=\Big|\
\frac{a r_k\rho_k^a(y) (\tilde y_k+r_k y)}{|\tilde y_k+r_ky|^2}\cdot \Big(F(\tilde x_k,\tilde y_k)-F(\tilde x_k,0)\Big)\Big|\\
&\le \frac{c r_k [F]_{C^{0,\alpha}(B_1)} |\tilde{y}_k|^\alpha}{|\tilde y_k|} \le cr_k |\tilde y_k|^{\alpha-1}.
\end{split}
\end{align}
Thus, by using \eqref{eq:div:F:radal} and \eqref{eq:II:term:case1}, one has that
\[
|\text{II}|\le \frac{c}{L_k} + \frac{ c r_k^{-\alpha}}{L_k} r_k|\tilde{y}_k|^{\alpha-1}\le \frac{c}{L_k}\Big(1+ \Big(\frac{r_k}{|\tilde{y}_k|}\Big)^{1-\alpha} \Big) \le cL_k^{-1}
\to 0,\quad \text{as }k \to\infty,
\]
in \textbf{Case 1} and \textbf{Case 2}, since $r_k/|\tilde{y}_k|\le c$.

\smallskip

Finally, we prove that the third member of \eqref{eq:vk:c1:limit} goes to zero as $k\to\infty$. In \textbf{Case 3}, $P_k=0$ so $\text{III}=0$. Instead, in \textbf{Case 1} and \textbf{Case 2}, one has
\begin{align}\label{eq:III:vanishes}
\begin{split}
|\text{III}|&
\le
\Big|
\frac{r_k^{-\alpha}}{L_k} \int_{\supp(\phi)}\rho_k^a(y) A(\tilde{z}_k)(\eta \D u_k)(\tilde{z}_k)\cdot \D\phi(z)dz
\Big|\\
& + \Big|
\frac{r_k^{-\alpha}}{L_k} \int_{\supp(\phi)}\rho_k^a(y) (A(\tilde z_k+r_kz)-A(\tilde{z}_k))(\eta \D u_k)(\tilde{z}_k)\cdot \D\phi(z)dz
\Big|.
\end{split}
\end{align}

First, we show that the first member in \eqref{eq:III:vanishes} vanishes. Recall that $r_k/|y_k|\to0$, $\rho_k^a\to 1$ and $(x_k,y_k^0)$ is the projection of $z_k$ on $\partial\Sigma_{\e_k}$. By using \eqref{eq:c1:l:infty} and \eqref{eq:stima:grad:bordo} we get 
\begin{align}\label{eq:III:i}
\begin{split}
&|(\eta \D u_k) (\tilde{z}_k))| \le  |\D (\eta u_k) (\tilde x_k, \tilde y_k)| + |(u_k \D \eta) (\tilde x_k, \tilde y_k)|\\
&\le 
|\D (\eta u_k)(\tilde x_k, \tilde y_k)-\D (\eta u_k)(x_k,y_k^0)|
+|\D (\eta u_k)(x_k,y_k^0)|\\
& + |\D \eta (\tilde x_k, \tilde y_k) (u_k(\tilde x_k, \tilde y_k) - u_k( x_k, y_k^0) )|\\
&\le 
cL_k |\tilde y_k-y_k^0|^\alpha+ cL_k \e_k^\alpha + c|\tilde{y}_k -y_k^0|^{\alpha}
\le c L_k |\tilde y_k|^\alpha.
\end{split}
\end{align}
On the other hand, since $|\tilde y_k+r_k y|\ge |\tilde y_k|/2$ for $y\in\supp(\phi)$, it follows
\begin{equation}\label{eq:III:ii}
|\D\rho_k^a(y)|=\Big|
\frac{a r_k\rho_k^a(y) (\tilde y_k+r_k y)}{|\tilde y_k+r_ky|^2}
\Big|
\le c\frac{r_k}{|\tilde y_k|}\rho_k^a(y).
\end{equation}
hence, by combining \eqref{eq:III:i} and \eqref{eq:III:ii}, and using the divergence theorem, it follows that
\begin{align*}
\Big|
\frac{r_k^{-\alpha}}{L_k} \int_{\supp(\phi)}\D \rho_k^a(y) \cdot A(\tilde{z}_k)(\eta \D  u_k)(\tilde{z}_k)\phi(z)dz
\Big|\le c\Big(\frac{r_k}{|\tilde{y}_k|}\Big)^{1-\alpha} \to 0,
\end{align*}
that is, the first member in \eqref{eq:III:vanishes} vanishes. Next, we show that the second member vanishes as $k\to\infty$. In this case, we need to reason in two steps (as done in \cite{SirTerVit21a}*{Remark 5.3} and \cite{AFV24}*{Theorem 7.1}): first, we prove uniform estimates in $C^{1,\alpha'}$ space for some suboptimal $\alpha' \in (0,\alpha)$; then, by using these estimates, we conclude the optimal regularity with exponent $\alpha$. Let us fix $\alpha' \in (0,\alpha)$. By using interpolation inequality in H\"older spaces \cite{GT01}*{Lemma 6.35}, we can estimate the second term of \eqref{eq:III:vanishes} as follows
\begin{align*}
&\Big|
\frac{r_k^{-\alpha'}}{L_k} \int_{\supp(\phi)}\rho_k^a(y) (A(\tilde z_k+r_kz)-A(\tilde{z}_k))(\eta \D  u_k)(\tilde{z}_k)\cdot \D\phi(z)dz
\Big|\\
& \le \frac{c r_k^{\alpha - \alpha'} \|(\eta \D  u_k)\|_{L^\infty(B_{3/4}\setminus\Sigma_{\e_k})} }{L_k} 
\le 
 \frac{c r_k^{\alpha - \alpha'}  }{L_k} \big(
\|u_k \D\eta  \|_{L^\infty(B_{3/4}\setminus\Sigma_{\e_k})}
+
\|\D (\eta u_k)  \|_{L^\infty(B_{3/4}\setminus\Sigma_{\e_k})} 
 \big) \\
&\le  \frac{c r_k^{\alpha - \alpha'}  ( \|  u_k \|_{L^\infty(B_{3/4}\setminus\Sigma_{\e_k})} + [ \D(\eta u_k) ]_{C^{0,\alpha}(B_{3/4}\setminus\Sigma_{\e_k})}) }{L_k} \le c r_k^{\alpha - \alpha'} \to 0,
\end{align*}
since $\alpha'<\alpha$. If we have uniform estimates in $C^{1,\alpha'}$ space, then $\|\eta \D u_k\|_{ {L^\infty(B_{3/4}\setminus\Sigma_{\e_k})} } \le c$. Hence, restarting the proof with the optimal $\alpha$ and the additional information above, in the previous computation we get
\begin{align*}
&\Big|
\frac{r_k^{-\alpha}}{L_k} \int_{\supp(\phi)}\rho_k^a(y) (A(\tilde z_k+r_kz)-A(\tilde{z}_k))(\eta \D  u_k)(\tilde{z}_k)\cdot \D\phi(z)dz
\Big|\\
& \le \frac{c  \|(\eta \D  u_k)\|_{L^\infty(B_{3/4}\setminus\Sigma_{\e_k})} }{L_k} \le \frac{c  }{L_k} \to 0.
\end{align*}
Combining all the previous results, we conclude that the right-hand side of \eqref{eq:vk:c1:limit} vanishes as $k \to \infty$.

\smallskip

Finally, by the same considerations of Theorem \ref{teo:0alpha:e}, we obtain that the left hand side of \eqref{eq:vk:c1:limit} converges in the following sense
\begin{equation*}
\int_{\supp(\phi)}\rho_k^a A_k \D w_k\cdot \D {\phi}dz\to \int_{\supp(\phi)}\bar{\rho}^a \bar{A}\D \bar{v}\cdot \D \phi dz, 
\end{equation*}
where $$\bar\rho(y):=\begin{cases}
1,&\text{ in }\textbf{Case 1 }\text{and } \textbf{Case 2},\\
|y|,&\text{ in }\textbf{Case 3},
\end{cases}$$
and $\bar{v}\in H^{1}_\loc(\Omega_\infty,\bar{\rho}^a(y)dz)$.

\smallskip

Then, recalling who is $\Omega_\infty$, see \eqref{eq:limit:domain}, and using the \emph{Step 5} for the Dirichlet boundary condition, we conclude that
\begin{itemize}[left=0pt]
\item in \textbf{Case 1}, $\bar{v}$ is an entire solution to
$$-\div(\bar{A}\D\bar{v})=0,\quad\text{in }\R^d,$$

\item in \textbf{Case 2}, $\bar{v}$ is an entire solution to

$$\begin{cases}
-\div(\bar{A}\D\bar{v})=0, & \text{in }\Pi ,\\
\bar{v}=0, & \text{on } \partial\Pi,,
\end{cases}$$
\item in \textbf{Case 3}, $\bar{v}$ is an entire solution to

$$\begin{cases}
-\div(|y|^a\bar{A}\D\bar{v})=0, & \text{in }\R^d\setminus\Sigma_{\bar{\e}},\\
\bar{v}=0, & \text{on } \partial\Sigma_{\bar{\e}},
\end{cases}$$
\end{itemize}

\medskip

\noindent
\emph{Step 7. Liouville Theorems and conclusion.}
Since $\bar{v}$ satisfies the growth condition \eqref{eq:eps:growth:1} with $1+\alpha \in (0,2) \cap (0,2-a-n)$, invoking the classical Liouville Theorem in \textbf{Case 1} and \textbf{Case 2} shows that $\bar{v}$ must be a linear function. In contrast, applying the Liouville Theorem \ref{teo:liouville} in \textbf{Case 3} implies that $\bar{v}$ must be zero.

This contradicts \emph{Step 3}, as $\D\bar{v}$ is not constant. Consequently, $L_k = [\D (\eta u_k)]_{C^{0,\alpha}(B_1 \setminus \Sigma_{\e_k})}$ must be bounded, which implies that $\|\eta u_\e\|_{C^{1,\alpha}(B_1 \setminus \Sigma_\e)} \le c$. Thus, \eqref{eq:stima:c1:e} holds true.

Furthermore, recalling \eqref{eq:stima:grad:bordo} and using \eqref{eq:stima:c1:e}, we conclude that \eqref{eq:formal:conormal:eps} follows. This completes the proof.
\end{proof}

\begin{oss}\label{R:F:radial}

The homogeneous Dirichlet condition on \(\partial \Sigma_\varepsilon\) is too restrictive to handle all possible fields \(F\), thus preventing the validity of \(\varepsilon\)-uniform \(C^{1,\alpha}\) estimates in the general case.

Let us suppose that $n=d=2$, $-2<a<-1$ and consider the function $u(y)=u(y_1,y_2):=y_1+y_2$ which is a weak solution to 
\[
\begin{cases}
-\div(|y|^a\D u)=\div(|y|^a F), & \text{in }B_1\setminus\Sigma_{{0}},\\
u=0, & \text{on } \Sigma_{0}\cap B_1.
\end{cases}
\]
where $F:=-\D u =(-1,-1)$ .
By applying Lemma \ref{eq:lemm:approximation:e} we find that there exists a family $u_\e$ which are weak solutions to 
\[
\begin{cases}
-\div(|y|^a\D u_\e)=\div(|y|^a F), & \text{in }B_{3/4}\setminus\Sigma_{{\e}},\\
u=0, & \text{on } \Sigma_{\e}\cap B_{3/4},
\end{cases}
\]
and satisfies $\|u_\e\|_{H^{1,a}(B_{3/4}\setminus\Sigma_\e)}\le c$ and $u_\e \to u$ in ${H^{1,a}(B_{3/4})}$. If Theorem \ref{teo:1alpha:e} holds true for this equation, we  obtain that \eqref{eq:formal:conormal:eps} works (since it depends only on the Dirichlet boundary condition satisfied by $u_\e$), that is
\[
|\D u_\e|\le c \e^\alpha,\quad \text{ on }\partial \Sigma_\e\cap B_{1/2}, 
\]
and then, taking the limit as $\e\to 0$ in a suitable sense (see the proof of the Theorem \ref{teo:1}), we deduce that $\D u =0$ on $\Sigma_0\cap B_{1/2}$, which contradicts $\D u = (1,1)$.

\end{oss}

The next proposition provides \emph{a priori} estimates for solutions that additionally satisfy an extra boundary condition on the lower dimensional boundary $\Sigma_0$.

\begin{pro}\label{P:a:priori}
Let $2\le n\le d$, $a+n\in(0,1)$, $p>{d}$ and $\alpha$ satisfying \eqref{eq:alpha:1}.
Let $A$ be a $\alpha$-H\"older continuous matrix satisfying \eqref{eq:unif:ell} and $\|A\|_{C^{0,\alpha}{(B_1)}}\le L$, $f\in L^{p,a}(B_1)$ and $F\in C^{0,\alpha}(B_1)$. Let $u \in C^{1,\alpha}(B_{1})$ be a weak solution to 
\[
\begin{cases}
-\div(|y|^a A \nabla u)= |y|^a f+\div(|y|^aF), & \text{in }B_1\setminus \Sigma_{0},\\
u=0, & \text{on } \Sigma_0\cap B_1,
\end{cases}
\]
such that $u$ satisfies the boundary condition \begin{equation}\label{eq:BC:conormal:a:priori}
\begin{cases}
\D_x u(x,0)=0, \\
(A\D u +F) (x,0) \cdot e_{y_i}= 0,
\end{cases} \text{ for every } (x,0)\in \Sigma_0\cap B_{1}, \text{ and } i=1,\dots,n.
\end{equation} 
Then, there exists a constant $c>0$, depending only on $d$, $n$, $a$, $\lambda$, $\Lambda$, $p$, $\alpha$ and $L$ such that 
\begin{equation}\label{eq:c1:a:priori}
\|u\|_{C^{1,\alpha}(B_{1/2})}  \le c\big(
\|u\|_{L^{2,a}(B_{1})}
+\|f\|_{L^{p,a}(B_{1})}
+\|F\|_{C^{0,\alpha}(B_1)}
\big).
\end{equation}
\end{pro}

\begin{proof}
The proof is quite similar to the one of Theorem \ref{teo:1alpha:e}, so we avoid some details. Without loss of generality, let us suppose that 
\[
\|u\|_{L^{2,a}(B_{1})}
+\|f\|_{L^{p,a}(B_{1})}
+\|F\|_{C^{0,\alpha}(B_1)} \le c.
\]
By contradiction let us suppose that \eqref{eq:c1:a:priori} doesn't holds, hence, there exist $p>d$, $\alpha$ satisfying \eqref{eq:alpha:1}, $\{A_k\}_k$, $\{f_k\}_k$, $\{F_k\}_k$, $\{u_k\}_k$, such that 
$u_k$ is solution to
\[
\begin{cases}
-\div(|y|^a A_k \nabla u_k)= |y|^a f_k+\div(|y|^aF_k), & \text{in }B_1\setminus \Sigma_{0},\\
u_k=0, & \text{on } \Sigma_0\cap B_1,
\end{cases}
\]
satisfies the boundary condition \eqref{eq:BC:conormal:a:priori} in $\Sigma_0\cap B_1$ and 
\[
\|u_k\|_{C^{1,\alpha}(B_{1/2})}\to  \infty.
\]
Let us fix a smooth cut-off function $\eta \in C_c^\infty(B_1)$ such that
\[
\supp(\phi)\subset B_{3/4},
\quad 
\eta=1 \text{ in }B_{1/2},
\quad
0\le\eta\le1, 
\]
hence, we have
\begin{equation*}
L_k:=[ \D (\eta u_k)]_{C^{0,\alpha}(B_{1})} \to \infty.
\end{equation*}
By definition of H\"older seminorm, take two sequences of points $z_k=(x_k,y_k),\hat{z}_k=(\hat x_k,\hat y_k)\in B_{1}$ such that
\begin{equation*}
\frac{|\D(\eta u_k)(z_k)-\D(\eta u_k)(\hat z_k)|}{|z_k-\hat z_k|^\alpha}\ge\frac{L_k}{2},
\end{equation*}
and define $r_k:=|z_k-\hat z_k|$. Now we distinguish two cases. 

   \medskip

$\bullet$ \textbf{Case 1:} $ 
    \displaystyle{
    \frac{|y_k|}{r_k}\to\infty,}\quad$ $\bullet$   \textbf{Case 2:}
    $ 
    \displaystyle{
    \frac{|y_k|}{r_k}\le c, }$
    
    \medskip
    \noindent
for some $c>0$ which not depends on $k$.
Let us define 
\[
\tilde z_k=(\tilde x_k,\tilde y_k):=\begin{cases}
(x_k,y_k), & \text{ in }\textbf{Case 1},\\
(x_k,0), & \text{ in }\textbf{Case 2},
\end{cases}
\]
the sequence of domains 
$$\Omega_ k:=\frac{B_{1}\setminus\Sigma_0-\tilde{z}_k}{r_k},$$
the sequences of functions
\begin{align*}
&v_k(z):= \frac{
(\eta u_k)( \tilde{z}_k+r_kz) -(\eta u_k )(\tilde{z}_k) - \D(\eta u_k)(\tilde{z}_k)\cdot r_k z
}{L_k r_k^{1+\alpha}},\\
&w_k(z):=\frac{\eta(\tilde{z}_k) u_k( \tilde{z}_k+r_kz) -  (\eta u_k) (\tilde{z}_k) - (\eta \D  u_k) (\tilde{z}_k)\cdot r_k z}{L_k r_k^{1+\alpha}},
\end{align*}
for $z \in \Omega_k$ and set $\Omega_\infty$ as in \eqref{eq:omega:infty}.

\smallskip

By following the argument in Theorem \ref{teo:1alpha:e}, \emph{Step 2}, we have that $\|v_k\|_{C^{1,\alpha}(K)} \le 1$ for every compact subset $K \subset \Omega_\infty$. Therefore, we can apply the Arzelà-Ascoli Theorem to conclude that $v_k \to \bar{v}$ in $C^{1,\gamma}(K)$ for every $\gamma \in (0,\alpha)$, and that
$|\bar{v}(z)| \le c(1 + |z|^{1+\alpha})$, for a.e. $z \in \Omega_\infty$.
Moreover, it follows that $w_k \to \bar{v}$ as well.

Next, by the same reasoning used in Theorem \ref{teo:1alpha:e}, \emph{Step 3}, we obtain that $\D \bar{v}$ is not constant. Continuing as in \emph{Step 4} of Theorem \ref{teo:1alpha:e}, we conclude that $r_k \to 0$, which implies that the limit domain $\Omega_\infty$ is given by
\[
\Omega_\infty:=\begin{cases}
\R^d, & \text{in }\textbf{Case 1},\\
\R^d\setminus\Sigma_0, & \text{in }\textbf{Case 2}.
\end{cases}
\]

Eventually, we prove that $\bar{v}$ is an entire solution to a homogeneous equation with constant coefficients and we reach a contradiction by invoking a Liouville type Theorem.
Since $\|A_k\|_{C^{0,\alpha}(B_1)}\le L$, we have that $A_k(\tilde{z}_k+r_k z)\to A(\bar{z}):=\bar{A}$, which is a constant matrix satisfying \eqref{eq:unif:ell}.
Let us define
\[
\rho_k(y):=\begin{cases}
\displaystyle{\frac{|{y}_k+r_k y|}{|y_k|}}, & \text{ in } \textbf{Case 1},\\
|y|, &  \text{ in } \textbf{Case 2}.
\end{cases}
\]
and fix $\phi\in C_c^\infty(\Omega_\infty)$. A straightforward computation show us that
\begin{align*}
&\int_{\supp(\phi)}\rho_k^a(y)A_k(\tilde{z}_k+r_k z)\D w_k(z)\cdot \D \phi (z)
dz
=
 \frac{\eta (\tilde{z}_k)r_k^{1-\alpha}}{L_k} \int_{\supp(\phi)}\rho_k^a(y) f_k(\tilde{z}_k+r_k z)\phi(z)dz
 \\
&- \frac{\eta (\tilde{z}_k) r_k^{-\alpha}}{L_k} \int_{\supp(\phi)}\rho_k^a(y) \big(
F_k(\tilde{z}_k+r_k z)-F_k(\tilde{z}_k) \big) \cdot\D \phi (z)dz\\
&- \frac{r_k^{-\alpha}}{L_k} \int_{\supp(\phi)}\rho_k^a(y)
\big(A_k(\tilde{z}_k+r_k z)-A_k(\tilde{z}_k)\big)(\eta \D u_k)(\tilde{z}_k)\cdot\D \phi(z)dz\\
& - \frac{r_k^{-\alpha}}{L_k}\rho_k^a(y)
\big(
\eta A_k\D  u_k + \eta F_k
\big) (\tilde{z}_k)\cdot\D\phi(z) dz =\text{I}+\text{II}+\text{III}+\text{IV}.
\end{align*}
The terms I, II, III vanishes as $k\to\infty$ exactly as in Theorem \ref{teo:1alpha:e}, \emph{Step 6}.

\smallskip

Next, we show that the fourth member goes to zero. By using the divergence theorem, we get 
\[
\int_{\supp(\phi)}\rho_k^a(y)
\big(\eta 
A_k\D  u_k + \eta F_k
\big) (\tilde{z}_k)\cdot\D\phi(z) dz = \int_{\supp(\phi)} \D \rho_k^a(y)
\cdot \big(\eta
A_k\D  u_k + \eta F_k
\big) (\tilde{z}_k)\phi(z) dz.
\]
In \textbf{Case 2}, 
since, $\tilde{z}_k=(x_k,0) \in\Sigma_0$, $u_k$ satisfies the boundary condition \eqref{eq:BC:conormal:a:priori} we have that 
$$\eta(\tilde{z}_k) (A_k\D u_k + F_k)(\tilde{z}_k)\cdot e_{y_i} =0, \quad\text{for every }i=1,\dots,n,$$ 
so IV$=0$.

Let us consider the \textbf{Case 1}, recalling that $\tilde{z}_k=z_k$, $r_k/|y_k|\to0$ and $\rho_k^a\to 1$.
Arguing as Theorem \ref{teo:1alpha:e}, \emph{Step 6}, we have that
$$
|\D \rho_k^a(y)|\le c\frac{r_k}{|y_k|}\rho_k^a(y),
$$
and by using the boundary condition \eqref{eq:BC:conormal:a:priori}
\begin{align*}
&\Big|(\eta A_k\D  u_k+\eta F_k)({z}_k)\cdot \frac{y_k+r_k y}{|y_k+r_k y|}\Big|
\\
&
\le \Big|\big((\eta A_k\D  u_k+\eta F_k)({z}_k)-(\eta A_k\D  u_k+\eta F_k)(x_k,0)\big)\cdot \frac{y_k+r_k y}{|y_k+r_k y|}\Big|\\
&\le [\eta A_k\D  u_k+\eta F_k]_{C^{0,\alpha}(B_{3/4})}|y_k|^\alpha
\le c L_k |y_k|^\alpha.
\end{align*}
Hence, combining these two inequalities we obtain
\[
\Big |\D \rho_k^a(y) \cdot (\eta A_k\D  u_k+\eta F_k)({z}_k)
\Big| \le c \frac{r_k}{|y_k|^{1-\alpha}}L_k,
\]
so
\begin{align*}
|\text{IV}|= \frac{r_k^{-\alpha}}{L_k} 
\Big| \int_{\supp(\phi)}\D\rho_k^a(y)\cdot \big(\eta A_k\D u_k +\eta F_k
\big) ({z}_k) \phi(z)  dz \Big|\le c\Big(\frac{r_k}{|y_k|}\Big)^{1-\alpha}\to0,
\end{align*}
as $k\to\infty$. 

On the other hand, arguing as in the \emph{Step 5} of the Theorem \ref{teo:0alpha:e}, we have that 
\begin{equation*}
\int_{\supp(\phi)}\rho_k^a(y) A_k(\tilde{z}_k+r_k z)\D w_k (z)\cdot \D {\phi}(z)dz\to \int_{\supp(\phi)}\bar{\rho}^a \bar{A}\D \bar{v}\cdot \D \phi dz,
\end{equation*}
where 
\[
\bar{\rho}(y):=\begin{cases}
1, & \text{ in }\textbf{Case 1},\\
|y|, & \text{ in }\textbf{Case 2}.
\end{cases}
\]
Then, we have that
\begin{itemize}[left=0pt]

\item in \textbf{Case 1}, $\bar{v}$ is an entire solution to
$$-\div(\bar{A}\D\bar{v})=0,\quad\text{in }\R^d,$$

\item in \textbf{Case 2}, $\bar{v}$ is an entire solution to

$$\begin{cases}
-\div(|y|^a\bar{A}\D\bar{v})=0, & \text{in }\R^d\setminus\Sigma_{{0}},\\
\bar{v}=0, & \text{on } \Sigma_{{0}},
\end{cases}.$$

\end{itemize}
where $\bar{A}$ is symmetric constant matrix satisfying \eqref{eq:unif:ell}.
From this point on, as in Theorem \ref{teo:1alpha:e}, \emph{Step 7}, by invoking appropriate Liouville type Theorems we get a contradiction and the thesis follows.
\end{proof}

\begin{proof}[Proof of the Theorems \ref{teo:1}]

As in the proof of Theorem \ref{teo:0}, without loss of generality, we may consider the function $u - \psi$ and provide the proof of the theorem for solutions to \eqref{eq:1} with homogeneous Dirichlet boundary conditions. Once this is established, the general case follows directly.
We divide the proof in two steps.

\medskip

\noindent
\emph{Step 1. } First, we claim that if $A \in C^{1,\alpha}(B_1)$ and $F \in C^{1,\alpha}(B_1)$ then $u \in C^{1,\alpha} (B_{1/2})$ and $u$ satisfies the boundary condition \eqref{eq:BC:conormal:a:priori} in $\Sigma_0\cap B_{1/2}$.

\smallskip

Let us split the matrix $A$ in blocks as follows
\[
A=\left(\begin{array}{c c}
       A_1&A_2  \\ 
       A_2^T&A_3 
  \end{array}\right),
\]
where $A_1 : B_R\to \R^{d-n,d-n}$, $A_2 : B_R\to \R^{d-n,n}$,  $A_3 : B_R\to \R^{n,n}$. Let us consider the decomposition $F=(F_x,F_y)$, where $F_x=(F_{x_1}, \dots, F_{x_{d-n}})$ and $F_y=(F_{y_1}, \dots, F_{y_{n}})$ and define the scalar function 
$$g(x,y)=A^{-1}_3(x,0) F_2(x,0)\cdot y,$$
which belongs to $C^{1,\alpha}(B_1)$, since the block $A_3$ satisfies the uniformly elliptic condition \eqref{eq:unif:ell}, and
$$
g(x,0)=0,\qquad \D g(x,0)=\big(0,(A_3^{-1}F_2 )(x,0)\big).
$$
The function $v:= u- g \in \tilde{H}^{1,a}(B_1)$ is a weak solution to 
\[
\begin{cases}
-\div(|y|^a A \nabla v)= |y|^a f+\div(|y|^a (F-A\D g)), & \text{in }B_1\setminus \Sigma_{0},\\
v=0, & \text{on } \Sigma_0\cap B_1,
\end{cases}
\]
where the field $F-A\D g \in C^\infty(B_{1})$ satisfies 
$$(F-A\D g)(x,0)\cdot e_{y_i}=0.$$
Arguing as in the proof of Theorem \ref{teo:0}, by applying Lemma \ref{lem:approximation}, we can find a sequence $\{v_{\e_k}\}$ as $\e_k\to0$, such that every $v_{\e_k}$ is solution to 
\[
\begin{cases}
-\div(|y|^a A \nabla v_{\e_k})= |y|^a f+\div(|y|^a(F-A\D g)), & \text{in }B_{3/4}\setminus \Sigma_{\e_k},\\
v_{\e_k}=0, & \text{on } \partial \Sigma_{\e_k}\cap B_{3/4},
\end{cases}
\]
and $v_{\e_k}\to v$ in $H^{1,a}(B_{3/4}\setminus \Sigma_0)$ as $\e_k\to0$. 
By applying the Theorem \ref{teo:1alpha:e} to the sequences $\{v_{\e_k}\}$, combined with the estimate \eqref{eq:stima:rhs}, we get that
\begin{equation}\label{eq:teo:1:a}
\|v_{\e_k}\|_{C^{1,\alpha}(B_{1/2}\setminus\Sigma_{\e_k})}\le c \big(
\|v\|_{L^{2,a}(B_{1})}
+\|f\|_{L^{p,a}(B_{1})}
+\|F-A\D g\|_{C^{0,\alpha}(B_1)}
\big),
\end{equation}
for some $c>0$ which not depends on $\e_k$.

By applying Arzelà-Ascoli Theorem we get that ${v}_{\e_k}\to w$ in $C_{\loc}^{1,\gamma}(B_{1/2}\setminus\Sigma_0)$, for every $\gamma\in(0,\alpha)$, and by the a.e. convergences $v_{\e_k}\to v$ it follows that $w=v$. Moreover, by taking $z,z'\in B_{1/2}\setminus\Sigma_0$ such that $z\not=z'$, we have that
\begin{align*}
&\frac{|\D v(z)-\D v(z')|}{|z-z'|^\alpha}=\lim_{\e_k\to0}\frac{|\D v_{\e_k}(z)-\D v_{\e_k}(z')|}{|z-z'|^\alpha}\\
&\le c \big(
\|v\|_{L^{2,a}(B_{1})}
+\|f\|_{L^{p,a}(B_{1})}
+\|F-A\D g\|_{C^{0,\alpha}(B_1)}
\big),
\end{align*}
which implies that 
\begin{align*}
&[v]_{C^{1,\alpha}(B_{1/2}\setminus\Sigma_0)}=\sup_{\substack{z,z'\in B_{1/2}\setminus\Sigma_0\\ z \not= z'}}\frac{|\D v(z)-\D v(z')|}{|z-z'|^\alpha}\\
&\le c \big(
\|v\|_{L^{2,a}(B_{1})}
+\|f\|_{L^{p,a}(B_{1})}
+\|F-A\D g\|_{C^{0,\alpha}(B_1)}
\big),
\end{align*}
By continuity, we can extend $v$ to the whole $B_{1/2}$ in such a way that $[v]_{C^{1,\alpha}(B_{1/2})} = [v]_{C^{1,\alpha}(B_{1/2} \setminus \Sigma_0)}$.
Combining the previous inequality with the $L^\infty_{\loc}$ bound of solutions (see Lemma \ref{lem:moser}) and interpolation inequality in H\"older spaces \cite{GT01}*{Lemma 6.35}, we get
\[
\|v\|_{C^{1,\alpha}(B_{1/2})}\le c (\|v\|_{L^\infty(B_{1/2})}+[v]_{C^{1,\alpha}(B_{1/2})})
\le c \big(
\|v\|_{L^{2,a}(B_{1})}
+\|f\|_{L^{p,a}(B_{1})}
+\|F-A\D g\|_{C^{0,\alpha}(B_1)}
\big).
\]
Hence, $v\in C^{1,\alpha}(B_{1/2})$, which immediately implies $u \in C^{1,\alpha}(B_{1/2})$.

\smallskip

Next, let us prove that $u$ satisfies the boundary condition \eqref{eq:BC:conormal:a:priori}.
Let us fix $z=(x,y)\in B_{1/2}\setminus\Sigma_0$ and let $z^{0}_k$ be the projection of $z$ on $\partial\Sigma_{\e_k}\cap B_{1/2}$. By using Theorem \ref{teo:1alpha:e} and \eqref{eq:formal:conormal:eps}, we get
\begin{align*}
&|\D v(z) |\le
|\D v(z)-\D v_{\e_k}(z)  |+
| \D v_{\e_k}(z)- \D v_{\e_k}(z^{0}_k)  |+
|  \D v_{\e_k}(z^{0}_k) |\\
&\le |\D v(z)-\D v_{\e_k}(z)  |
+
[v_{\e_k}]_{C^{1,\alpha}(B_{1/2}\setminus\Sigma_{\e_k})}|z-z^{0}_k|^\alpha
+
{c} [v_{\e_k}]_{C^{1,\alpha}(B_{1/2}\setminus\Sigma_{\e_k})}\e_k^\alpha\\
&\le o(1)+c \big(\|v\|_{L^{2,a}(B_{1})}+\|f\|_{L^{p,a}(B_1)}+ \|F-A\D g\|_{C^{0,\alpha}(B_{1})} \big) |y|^\alpha,\quad\text{as }\e_k\to0.
\end{align*}
By taking the limit as $|y|\to 0$ in the previous inequality, we conclude that  $\D v = 0$ on $\Sigma_0\cap B_{1/2}$ and,
by construction, it follows that
\[
0=\D_x v(x,0) = \D_x u(x,0)-\D_x g(x,0) = \D_x u(x,0), 
\]
and 
\[
0=(A\D v)(x,0)\cdot e_{y_i}= (A\D u-A\D g)(x,0)\cdot e_{y_i}=(A\D u+F)(x,0)\cdot e_{y_i},
\]
that is, \eqref{eq:BC:conormal:a:priori} holds true.

\medskip
\noindent
\emph{Step 2.}  Finally, we prove that if $A,F \in C^{0,\alpha}(B_1)$, then $u \in C^{1,\alpha}(B_{1/2})$ and satisfies \eqref{eq:c1}, \eqref{eq:BC:conormal}.
Up to consider the function $u-\psi \in\tilde{H}^{1,a}(B_1)$, we can suppose that $u$ satisfies a Dirichlet homogeneous boundary condition $u=0$ on $\Sigma_0\cap B_1$.

Let $\{\rho_\delta\}_{\delta>0}$ be a family of smooth mollifiers and define $A_\delta := A \ast \rho_\delta$ and $F_\delta := F \ast\rho_\delta $, which satisfies $\|A_\delta\|_{C^{0,\alpha}(B_{3/4})}\le \|A\|_{C^{0,\alpha}(B_{1})}$ and $\|F_\delta\|_{C^{0,\alpha}(B_{3/4})}\le \|F\|_{C^{0,\alpha}(B_{1})}$.  Using an approximation argument similar to the one in Lemma \ref{lem:approximation}, but with a more standard approach, and following a method analogous to \cite{AFV24}*{Theorem 4.1} in a similar context, we construct a family $\{u_\delta\}_{\delta > 0}$ of solutions to 
\[
\begin{cases}
-\div(|y|^a A_\delta \nabla u_\delta)= |y|^a f+\div(|y|^a F_\delta), & \text{in }B_{4/5}\setminus \Sigma_{0},\\
u=0, & \text{on } \Sigma_0\cap B_{4/5},
\end{cases}
\]
satisfies the following properties
\[
\|u_\delta\|_{H^{1,a}(B_{4/5})}\le c (\|u\|_{H^{1,a}(B_1)}+\|f\|_{L^{2,a}(B_1)}+\|F\|_{L^{2,a}(B_1)}),
\]
\[
u_\delta\to u \text{ strongly in } H^{1,a}(B_{4/5}),
\]
for some $c>0$ depending only on $d$, $n$, $a$, $\lambda$ and $\Lambda$.

Since $A_\delta, F_\delta \in C^\infty(B_{4/5})$, by using \emph{Step 1}, we get that $u \in C^{1,\alpha} (B_{3/4})$ and $u$ satisfies the boundary condition \eqref{eq:BC:conormal:a:priori} on $\Sigma_0\cap B_{3/4}$.
Hence, the assumptions of the Proposition \ref{P:a:priori} are satisfies, which implies that $u_\delta$ satisfies \eqref{eq:c1:a:priori}, which immediately implies that
$$\|u_\delta\|_{C^{1,\alpha}(B_{1/2})}\le c (\|u\|_{L^{2,a}(B_1)}+\|f\|_{L^{2,a}(B_1)}+\|F\|_{L^{2,a}(B_1)}), $$
and the Arzel\'a-Ascoli Theorem allows us to taking the limit as $\delta\to 0$ to infer that $u \in C^{1,\alpha}(B_{1/2})$, satisfies \eqref{eq:c1} and the boundary condition \eqref{eq:BC:conormal}.
The proof is complete.
\end{proof}

\section{Regularity on curved manifolds}\label{section:curved}

In this section, we prove how to extend Theorems \ref{teo:0} and \ref{teo:1}, namely the local $C^{0,\alpha}$ and $C^{1,\alpha}$ regularity of weak solutions to \eqref{eq:1}, to a class of equations with weights that are singular on $(d-n)$-dimensional manifolds, that is, for weak solutions to \eqref{eq:1:cor}.

\smallskip

For $2\le n<d$, let $\Gamma\subset\R^d$ be a $(d-n)$-dimensional $C^1$-manifold such that $0\in\Gamma$ and there exists a parametrization $\varphi\in C^1(\Sigma_0\cap B_1;\R^n)$ such that, up to perform a dilation, we have 
\begin{equation}\label{eq:parametrization}
B_1\cap \Gamma=\{(x,y):y =\varphi(x)\}\cap B_1,\quad\varphi(0)=0.
\end{equation}
The diffeomorphism
\begin{equation}\label{eq:diffeomorphism}
    \Phi(x,y):=(x,y+\varphi(x)),
\end{equation}
whose inverse straightens the lower dimensional boundary $\Gamma$ to $\Sigma_0$, satisfies $\Phi(\Sigma_0\cap B_1)\subset\Gamma\cap B_1$ and the Jacobian associated to $\Phi$ is
\[
J_\Phi(x,y)=\left(\begin{array}{c c}
       {I}_{d-n} & { 0}  \\ 
       J_\varphi(x) & {I}_{n}\\ 
  \end{array}\right), \quad \text{with } |\det J_\Phi|\equiv1.
\]
Given a $(d-n)$-dimensional manifold $\Gamma$ parametrized by $\varphi$, we give the definition of admissible weights with respect to $\varphi$.

\begin{defi}\label{def:rho}
Let $2\le n <d$ and $\alpha\in[0,1)$.
Let $\Gamma$ be a $(d-n)$-dimensional $C^{1,\alpha}$-manifold and $\varphi\in C^{1,\alpha}(\Sigma_0\cap B_1;\R^n)$ be a parametrization of $\Gamma$ in the sense of \eqref{eq:parametrization} and define the diffeomorphism $\Phi$ as in \eqref{eq:diffeomorphism}.

We say that $\delta$ is an $\alpha$-admissible weight with respect to the parametrization $\varphi$ if $\delta \in C^{0,1}(B_1)$ and the two following condition holds true.
\begin{itemize}
\item[i)] There exist constants $0< c_0 \le c_1$ such that \begin{equation}\label{eq:distance} 
c_0\le \frac{ \delta}{\dist_\Gamma}  \le c_1.
\end{equation}
\item[ii)] 
\begin{equation}\label{eq:rho:delta}
\tilde \delta (x,y):=\frac{\delta(\Phi(x,y))}{|y|}\in C^{0,\alpha}(B_1).
\end{equation}
\end{itemize}

\end{defi}

We point out that the condition i) in the Definition \ref{def:rho} implies that $\delta$ in a equivalent distance from $\Gamma$ to the standard distance function $\dist_\Gamma$. So, given $\delta$, which satisfies condition $i)$, similarly to what was done in Section \ref{section:sobolev}, we can define the functional spaces $H^1 (B_1,\delta^a)$, which is equivalent to $H^1 (B_1,\dist_\Gamma^a)$. Then, by Theorem \ref{teo:NEK}, it follows that it is well defined the bounded linear trace operator $T:H^1 (B_1,\delta^a)\to L^2(\Gamma\cap B_1)$ such that $Tu=u_{|_{\Gamma\cap B_1}}$, for every $u\in C^{\infty}(\overline{B_1})$.
In light of this, we define a notion of weak solutions for the equation \eqref{eq:1:cor}, which satisfy a Dirichlet boundary on $\Gamma$.

\begin{defi}\label{def:weak:sol:curva}
Let $2\le n< d$, $a+n\in(0,2)$, $A$ be a matrix satisfying \eqref{eq:unif:ell}. Let $\Gamma$ be a $(d-n)$-dimensional $C^1$-manifold and $\varphi\in C^1(\Sigma_0\cap B_1;\R^n)$ be a parametrization of $\Gamma$ in the sense of \eqref{eq:parametrization}, $\delta \in C^{0,1}(B_1)$ satisfying i) of the Definition \ref{def:rho}.
Let $f\in L^2(B_1,\delta^a)$, $F\in L^2(B_1,\delta^a)^d$ and $\psi \in L^2(\Gamma\cap B_1)$. 

We say that $u$ is a weak solution to
\eqref{eq:1:cor} if $u\in H^1 (B_1,\delta^a)$, satisfies
\begin{equation}\label{eq:weak:sol:rho}
\int_{B_1}\delta^a A\D u\cdot \D\phi dz= \int_{B_1}\delta^a (f\phi-F\cdot\D\phi) dz,
\end{equation}
for every $\phi\in C_c^\infty(B_1\setminus\Gamma)$ and $u=\psi$ in $L^2(\Gamma\cap B_1)$, in the sense of the trace.
\end{defi}

Finally, we can state the main results of this section, namely $C^{0,\alpha}_\loc$ and $C^{1,\alpha}_\loc$ regularity for weak solutions to \eqref{eq:1:cor}.

\begin{cor}\label{cor:0}
Let $2\le n < d$, $a+n\in(0,2)$, $p>{d}/{2}$, $q>d$ and $\alpha$ satisfying \eqref{eq:alpha:0}. Let $\varphi\in C^1(\Sigma_0\cap B_1;\R^n)$ be the parametrization defined in \eqref{eq:parametrization}. Let $\delta$ be a $0$-admissible weight with respect to $\varphi$ in the sense of Definition \ref{def:rho} and define $\tilde{\delta}\in C^{0}(B_1)$ as in \eqref{eq:rho:delta}.
Let $A$ be a continuous matrix satisfying \eqref{eq:unif:ell}, $f\in{L^{p}(B_{1},\delta^a)} $, $F\in{L^{q}(B_{1},\delta^a)}^d $ and $\psi \in C^{0,1}(\Gamma\cap B_1)$. Let $u$ be a weak solution to \eqref{eq:1:cor}, in the sense of Definition \ref{def:weak:sol:curva}. Let us suppose that $$\|A\|_{C^0 (B_1)}+\|\tilde \delta\|_{C^{0} (B_1)}+\|\varphi\|_{C^1 (\Sigma_0\cap B_1;\R^n)}\le L.$$

Then, $u\in C^{0,\alpha}(B_{1/2})$ and there exists a constant $c>0$, depending only on $d$, $n$, $a$, $\lambda$, $\Lambda$, $p$, $q$, $\alpha$ and $L$ such that
\begin{equation}\label{eq:cor:c0}
\|u\|_{C^{0,\alpha}(B_{1/2})}  \le c\big(
\|u\|_{L^{2}(B_{1},\delta^a)}
+\|f\|_{L^{p}(B_{1},\delta^a)}
+\|F\|_{L^{q}(B_{1},\delta^a)^d}
+\|\psi\|_{C^{0,1}(\Gamma \cap B_{1})}
\big).
\end{equation}

\end{cor}

\begin{cor}\label{cor:1}
Let $2\le n < d$, $a+n\in(0,1)$, $p>{d}$ and $\alpha$ satisfying \eqref{eq:alpha:1}. Let $\varphi\in C^{1.\alpha}(\Sigma_0\cap B_1;\R^n)$ be the parametrization defined in \eqref{eq:parametrization}. Let $\delta$ be a $\alpha$-admissible weight with respect to $\varphi$ in the sense of Definition \ref{def:rho} and define 
$\tilde \delta  \in C^{0,\alpha}(B_1)$, as in \eqref{eq:rho:delta}. 
Let $A$ be a $\alpha$-H\"older continuous matrix satisfying \eqref{eq:unif:ell}, $f\in{L^{p}(B_{1},\delta^a)} $, $F\in{C^{0,\alpha}(B_{1})} $, $\psi \in C^{1,\alpha}(\Gamma\cap B_1)$ and $u$ be a weak solution to \eqref{eq:1:cor}, in the sense of Definition \ref{def:weak:sol:curva}.
Let us suppose that $$\|A\|_{C^{0,\alpha} (B_1)}+\|\tilde\delta\|_{C^{0,\alpha} (B_1)}+\|\varphi\|_{C^{1,\alpha} (\Sigma_0\cap B_1;\R^n)}\le L.$$

Then, $u\in C^{1,\alpha}(B_{1/2})$ and there exists a constant $c>0$, depending only on $d$, $n$, $a$, $\lambda$, $\Lambda$, $p$, $\alpha$ and $L$ such that
\begin{equation}\label{eq:cor:c1}
\|u\|_{C^{1,\alpha}(B_{1/2})}  \le c\big(
\|u\|_{L^{2}(B_{1},\delta^a)}
+\|f\|_{L^{p}(B_{1},\delta^a)}
+\|F\|_{C^{0,\alpha}(B_{1})}
+\|\psi\|_{C^{1,\alpha}(\Gamma \cap B_{1})}
\big).
\end{equation}

Moreover, denoting by $T_z\Gamma$ the tangent space to $\Gamma$ at the point $z\in\Gamma$, we have that $u$ satisfies the following boundary condition for every $z\in \Gamma\cap B_{1/2}$,
\begin{equation}\label{eq:BC:curve}
\begin{cases}
\D u(z)\cdot \tau(z) =\D\psi(z)\cdot \tau(z),& \text{ for every }\tau(z)\in T_z{\Gamma} , \\
(A\D u +F) (z) \cdot \nu(z)= 0,&  \text{ for every } \nu(z)\perp T_z{\Gamma} .
\end{cases}
\end{equation}

\end{cor}

Since the proofs of Corollaries \ref{cor:0} and \ref{cor:1} are quite similar, we will only provide the proof of the second one, as it is more involved.

\begin{proof}[Proof of Corollary \ref{cor:1}]
Let us define the diffeomorphism $\Phi$ as in \eqref{eq:diffeomorphism}, 
which is of class $C^{1,\alpha}(B_1)$, by the assumption $\varphi\in C^{1,\alpha}(\Sigma_0\cap B_1)$. 
Defining
$$\tilde{u}(x,y):=u\circ \Phi(x,y),\qquad \tilde\psi(x)=\psi \circ \Phi(x,0) \in C^{1,\alpha}(\Sigma_0\cap B_1).$$
we have that $\tilde{u} \in H^{1,a}(B_1)$ and $\tilde u=\tilde \psi$ on $\Sigma_0\cap B_1$ in the sense of the traces.
Since $\delta$ is an $\alpha$-admissible weight in the sense of Definition \ref{def:rho}, the conditions \eqref{eq:distance} and \eqref{eq:rho:delta} implies that
\begin{equation}\label{eq:rho:B:*}
\tilde\delta^a (x,y) = \frac{\delta(\Phi(x,y))^a}{|y|^a} \in C^{0,\alpha}(B_1),\quad \tilde\delta^a\ge \tilde{c}_0>0,
\end{equation}
for some constant $\tilde c_0>0$, where $\tilde{\delta}=\delta\circ \Phi$.

Let $\phi\in C_c^\infty(B_1\setminus\Gamma)$ be a test function in \eqref{eq:weak:sol:rho}. By taking the change of variables $z=\Phi(x,y)$ it follows that
\begin{equation}\label{eq:int:curve}
0=\int_{B_1}\delta^a \big(A\D u\cdot\D\phi-f\phi+F\cdot\D\phi\big) dz
=\int_{B_1}|y|^a\big(\tilde A\D\tilde u\cdot\D\tilde\phi-\tilde f\tilde\phi+\tilde F\cdot\D\tilde \phi\big) dz,
\end{equation}
where 
\begin{align}\label{eq:tilde:u:f:A}
\tilde{A}:=  \tilde{\delta}^a (J_\Phi^{-1})(A\circ \Phi)(J_\Phi^{-1})^T,
\quad  \tilde{f}:= \tilde{\delta}^a f\circ\Phi,\quad \tilde{F}:=  \tilde{\delta}^a (J_\Phi^{-1}) F\circ\Phi,\quad \tilde{\phi}:= \phi\circ\Phi.
\end{align} 
By using \eqref{eq:rho:B:*}, we have that
$$\tilde{A}\in C^{0,\alpha}(B_1) \text{ and satisfies }\eqref{eq:unif:ell},\quad \tilde{f} \in  L^{p,a}(B_1),\quad \tilde{F}\in C^{0,\alpha}(B_1).$$
Hence, we have proved that $\tilde{u}$ is a weak solution to
\begin{equation*}
\begin{cases}
-\div(|y|^a \tilde{A} \nabla \tilde{u})= |y|^a \tilde{f}+\div(|y|^a\tilde{F}), & \text{in }B_1\setminus \Sigma_{0},\\
\tilde u=\tilde{\psi}, & \text{on } \Sigma_0 \cap B_1,
\end{cases}
\end{equation*}
and $\tilde u$ satisfies the hypothesis of the Theorem \ref{teo:1}.
Hence, $\tilde u$ satisfies \eqref{eq:c1} and composing back with the diffeomorphism $\Phi^{-1}$ we get that $u$ satisfies \eqref{eq:cor:c1}.

\smallskip

Eventually, let us prove that $u$ satisfies the boundary condition \eqref{eq:BC:curve}.
By Theorem \ref{teo:1}, we have that $\tilde u$ satisfies the boundary condition
\begin{equation}\label{eq:BC:bar:u}
\D_x \tilde u=\D_x\tilde \psi, \quad
(\tilde A\D \tilde u +\tilde F) \cdot e_{y_i}= 0,\quad
 \text{on } \Sigma_0\cap B_{1/2}, \text{ for every } i=1,\dots,n.
\end{equation}
Since $\tilde{u}(z)=u(\Phi(z))$, one has that 
\[
\D \tilde{u}(z) = J_\Phi^T(z) \D u(\Phi(z)).
\]
and, noting $\Phi(x,0)=(x,\varphi(x))\in \Gamma$, we have that
$$
\D \tilde{u}(x,0)  \cdot e_{x_j} = J_\Phi^T(x,0) \D u(\Phi(x,0))
\cdot e_{x_j} =  \D u(x,\varphi(x))\cdot J_\Phi(x,0) e_{x_j}.
$$
Thus, for every $j=1,\dots,d-n$, and $(x,0)\in \Sigma_0\cap B_{1/2}$, we have
\begin{equation}\label{eq:tau:x:curve}
\D u(x,\varphi(x)) \cdot J_\Phi(x,0) e_{x_j} = \D \psi (x,\varphi(x)) \cdot J_\Phi(x,0) e_{x_j}.
\end{equation}
Next, recalling \eqref{eq:rho:B:*}, \eqref{eq:tilde:u:f:A} and \eqref{eq:BC:bar:u}, we find
\begin{align}\label{eq:nu:y:curve}
\begin{split}
0=\tilde{\delta}^{-a}(\tilde {A}\D \tilde  u +\tilde F)(x,0) \cdot e_{y_i} &=  J_\Phi^{-1}(x,0)(A\D u +F)(x,\varphi(x))\cdot e_{y_i} \\
&= (A\D u +F)(x,\varphi(x))\cdot (J_\Phi^{-1})^T(x,0) e_{y_i}.
\end{split}
\end{align}
Finally, observing that the tangent space to $\Gamma$ at the point $z=(x,\varphi(x))$ is given by
$$T_{(x,\varphi(x))}\Gamma :=\{(\xi,J_\varphi(x) \xi ): \xi\in \R^{d-n}\},$$ 
and noting that
$$
J_\Phi(x,0) e_{x_j}\cdot  (J_\Phi^{-1})^T(x,0) e_{y_i} = 0, \quad \text{for every } j={1,\dots,d-n},\quad i=1,\dots,n,
$$
it follows that \eqref{eq:tau:x:curve} and \eqref{eq:nu:y:curve} implies that \eqref{eq:BC:curve} holds true. This complete the proof.
\end{proof}

\section*{Acknowledgements} 
I would like to thank Gabriele Cora, Susanna Terracini and Stefano Vita for many very helpful suggestions on the paper.

\end{document}